\documentclass[12pt]{amsart}
\usepackage{graphicx}
\usepackage{amssymb}
\usepackage{amscd}
\usepackage{mathrsfs}
\usepackage{txfonts}
\newtheorem{theorem}{Theorem}[section]
\newtheorem{definition}[theorem]{Definition}
\newtheorem{proposition}[theorem]{Proposition}
\newtheorem{lemma}[theorem]{Lemma}
\newtheorem{corollary}[theorem]{Corollary}

\newtheorem{fact}[theorem]{Fact}
\newtheorem{problem}[theorem]{Problem}
\newtheorem{conjecture}[theorem]{Conjecture}
\theoremstyle{remark}
\newtheorem{remark}[theorem]{Remark}
\theoremstyle{definition}
\newtheorem{example}[theorem]{Example}
\newcommand{\F}{\mathbb F}
\newcommand{\R}{\mathbb R}

\newcommand{\C}{\mathbb C}
\newcommand{\Q}{\mathbb Q}
\newcommand{\Z}{\mathbb Z}
\newcommand{\D}{\Delta}
\newcommand{\GL}{\mathit{GL}}

\newcommand{\SL}{\mathit{SL}}
\newcommand{\tr}{\mathit{tr}}
\begin{document}

\title[twisted Alexander polynomial]
{Introduction to twisted Alexander polynomials and related topics}

\author{Teruaki Kitano}

\address{Department of Information Systems Science, 
Faculty of Science and Engineering, 
Soka University, 
Tangi-cho 1-236, 
Hachioji, Tokyo 192-8577, Japan}

\email{kitano@soka.ac.jp}


\maketitle
%%%%%%%%%%%%
\section{Introduction}

This article is based on the lectures in the Winter Braids V (Pau, February 2015). 
One purpose of these lectures was to explain how to compute twisted Alexander polynomials for non-experts. 
For this purpose we treated only twisted Alexander polynomials for knots 
and discussed many concrete examples. 
It is also keeping in this article. 
The author intended to write concrete computations in this article to be self-contained. 

There are two good survey papers \cite{Friedl-Vidussi11-2, Morifuji15} on this subjects. 
Since this article is more elementary, 
then we recommend to read them for more advanced topics.

First we recall there are many definitions (many faces) of the classical Alexander polynomial: 

\begin{itemize}
\item
Seifert form on a Seifert surface.
\item
Fox's free differentials to a presentation of a knot group.
\item
an order of the Alexander module (an infinite cyclic covering).
\item
Reidemeister torsion.
\item
Burau representation of the braid group.
\item
Obstruction to deform an abelian representation into non commutative direction.
\item
Skein relation.
\item
Euler characteristic of the knot Floer homology.
\end{itemize}

We can generalize some of them to twisted Alexander polynomials. 
\begin{itemize}
\item
Lin defined twisted Alexander polynomial for a knot by using a Seifert surface. 
\item
Wada also defined it for a finitely presentable group by using Fox's free differential. 
\item
Jang and Wang generalized the Lin's idea to other invariants. 
\item
Kirk and Livingston organized each of these perspectives, 
in particular, an order of the Alexander module. 
This is also related with an infinite cyclic covering.
\item
Twisted Alexander polynomial of a knot can be described as the Reidemeister torsion of its knot exterior. 
\end{itemize}

From each position of these studies we have slightly different invariants, 
but essentially the same one, 
which are called twisted Alexander polynomials. 
In this lecture note, we mainly follow the definition of the twisted Alexander polynomial by Wada. 
Twisted Alexander polynomial (Wada's invariant) can be defined 
for a finitely presentable group 
with an epimorphism onto a free abelian group. 
For simplicity, we discuss this invariant only for a knot group 
with the abelianization. 

\flushleft\textbf{Acknowledgements}:
The author was partially supported by JSPS KAKENHI 25400101. 
He would like to express my sincere gratitude 
to the organizers of Winter braid V;
Paolo Bellingeri, Vincent Florens, Jean-Baptiste Meilhan, Emmanuel Wagner. 

The author stayed in Aix-Marseille University when this article was wrote up. 
He also thanks friends of this university for their hospitality. 

%%%%%%%%%%%%%%%%%%%%%%%%%%%%%
\section{Fox's free differentials}

To define the Alexander polynomial we need one algebraic tool. 
It is the Fox's free differentials. See \cite{Fox, Crowell-Fox} as a reference.

\begin{definition}
An integral group ring of a group $G$ is a ring given by 
\[
\Z G=\{ 
\text{a finite formal sum }\sum_{g\in G} n_g g\ |\ n_g\in\Z\}
\]
as a set. 
Here finite  means the number of $n_g\neq 0$ is finite. 
The two operations of a group ring are defined by the following; 
\begin{itemize}
\item
sum:
$\displaystyle
\sum_{g\in G} n_g g+\sum_{g\in G} m_g g=\sum_{g\in G} (n_g+m_g) g$.
\item
multiplication:
$\displaystyle
\sum_{g\in G} n_g g\cdot\sum_{g\in G} m_g g=\sum_{g\in G}\left(\sum_{h\in G} n_h\cdot m_{h^{-1}g}\right) g$.
\end{itemize}
\end{definition}

%%%%%%%%%%%%%%%%%%
\begin{remark}\noindent
\begin{itemize}
\item
The unit of $\Z G$ as a group ring is $1=1(\in\Z)\times 1(\in G)$.
\item
We can define a group ring of $G$ over $\Q,\R, \C$, 
and write respectively $\Q G$, $\R G$ and $\C G$ for them.
\end{itemize}
\end{remark}

\begin{example}$\Z=\langle t \rangle$

For any element of $\Z \Z=\Z\langle t\rangle$, 
it is a form of 
$\displaystyle\sum_{k\in\Z} n_k t^k$. 
This can be considered as a Laurent polynomial of $t$. 
From now we always identify the group ring 
$\Z\Z=\Z\langle t\rangle$ 
with the Laurent polynomial ring $\Z[t,t^{-1}]$.
\end{example}

Let $F_n=\langle x_1,\cdots,x_n\rangle$ 
be the free group generated by $\{x_1,\cdots,x_n\}$. 
Fox's free differentials are algebraic derivations on $\Z F_{n}$. 

\begin{definition}
Fox's free differentials are maps
\[
\frac{\partial}{\partial x_1},\cdots,\frac{\partial}{\partial x_n}
:\Z F_n\rightarrow\Z F_n
\] 
satisfying the following conditions:
\begin{enumerate}
\item
They are linear over $\Z$. 
\item
For any $i,j$, $\displaystyle\frac{\partial}{\partial x_j}(x_i)=\delta_{ij}=\begin{cases} &1\ (i=j),\\ &0\ (i\neq j).\end{cases}$
\item
For any $g,g'\in F_n$, 
$\displaystyle\frac{\partial}{\partial x_j}(gg')
=\displaystyle\frac{\partial}{\partial x_j}(g)
+g\displaystyle\frac{\partial}{\partial x_j}(g')$.
\end{enumerate}
\end{definition}

\begin{lemma}
The followings hold;
\begin{itemize}
\item
$\displaystyle\frac{\partial}{\partial x_j}(1)=0$. 
\item
$\displaystyle\frac{\partial}{\partial x_j}(g^{-1})
=-g^{-1}\frac{\partial}{\partial x_j}(g)$ for any $g\in F_n$.
\item
$\displaystyle\frac{\partial}{\partial x_j}(x_j^k)
=
\begin{cases}
&\ 1+x_j+\dots +x_j^{k-1}\ (k>0), \\
& -(x_j^{-1}+\dots +x_j^{k})\ \ (k<0).
\end{cases}
$
\item
For any $g\in F_n$,
$\displaystyle
\frac{\partial}{\partial x_j}(g^k)
=
\begin{cases}
&\ \displaystyle\frac{g^k-1}{g-1}\frac{\partial}{\partial x_j} (g)\ \ (k>0),\\
&\displaystyle-\frac{g^k-1}{g-1}\frac{\partial}{\partial x_j} (g)\ \ (k<0).
\end{cases}
$
\end{itemize}
\end{lemma}

For simplicity,  
we frequently write $\displaystyle\frac{\partial w}{\partial x_{i}}$ 
to $\displaystyle\frac{\partial}{\partial x_{i}}(w)$ 
for any $w\in\Z F_{n}$. 

The following formula is the algebraic version 
of a linear approximation in the group ring of a free group. 

\begin{proposition}[Fundamental formula of free differentials]
For any $w\in \Z F_n$, 
it holds that 
\[
w-1=\sum_{j=1}^n \frac{\partial w}{\partial x_j}(x_j-1).
\]
\end{proposition}

\begin{proof}

We prove this formula by the induction on the word length $l(w)$ of $w\in F_n$.

For the case of $l(w)=0$, that is, $w=1$, 
it is clear that $w-1=0$ 
and $\displaystyle\sum_{j=1}^n \frac{\partial w}{\partial x_j}(x_j-1)=0$.

Assume it is true for any word $w$ with $l(w)=k$. 
Take any $w\in F_n$ with $l(w)=k+1$. 
We may assume $w=w_k x_i^{\pm 1}$ with $l(w_k)=k$. 
If $w=w_k x_i$, then one has 
\[
\begin{split}
\sum_{j=1}^n \frac{\partial w}{\partial x_j}(x_j-1)
&=\sum_{j=1}^n \frac{\partial (w_k x_i)}{\partial x_j}(x_j-1)\\
&=\sum_{j=1}^n 
\left(\frac{\partial w_k}{\partial x_j}+w_k\delta_{i,j}\right)(x_j-1)\\
&=\sum_{j=1}^n 
\frac{\partial w_k}{\partial x_j}(x_j-1) 
+w_k (x_i-1).
\end{split}
\]
By the assumption on the induction, 
\[
\sum_{j=1}^n 
\frac{\partial w_k}{\partial x_j}(x_j-1) =w_k-1.
\]
Hence we obtain 
\[
\begin{split}
\sum_{j=1}^n \frac{\partial w}{\partial x_j}(x_j-1)
&=
\sum_{j=1}^n 
\frac{\partial w_k}{\partial x_j}(x_j-1) 
+w_k( x_i-1)\\
&=w_k-1+w_k(x_i-1)\\
&=w_k x_i-1\\
&=w-1.
\end{split}
\]

Similarly it can be proved for the case of $w=w_k x_i^{-1}$. 

Further it can be done for any $w\in\Z F_n$ 
by using the linearity of free differentials.

This completes the proof.
\end{proof}
%%%%%%%%%%%%%%%%%%%%%%%%%%%
\section{Alexander polynomials}

In this section we apply the Fox's free differentials to get a knot invariant as follows. 
We put \cite{Burde-Zieschang-Heusener, Rolfsen} for terminologies and definitions of a knot theory as references. 

\subsection{definition}

Let $K\subset S^3$ a knot in $S^{3}$ and $G(K)=\pi_{1}(S^{3}-K)$ the knot group of $K$. 
We take and fix a presentation of $G(K)$ as 
\[
G(K)=\langle x_1,\ldots,x_n\ |\ r_1,\ldots,r_{n-d}\rangle. 
\]
Now we do not assume it is a Wirtinger presentation. 

For simplicity we explain first how to define the invariant 
for the case of $d=1$. 
Here the number $d$ is called the deficiency of a finite presented group, 
which is defined by 
the number of generators minus the number of relators.

Let us take a presentation of deficiency one as 
\[
G(K)=\langle x_1,\ldots,x_n\ |\ r_1,\ldots,r_{n-1}\rangle. 
\]
By using the above fixed presentation, 
an epimorphism 
\[
F_n\ni x_i\mapsto x_i\in G(K)\] 
is naturally defined. 
Further we consider a ring homomorphism 
\[
\Z F_n\to\Z G(K)
\]
induced from this epimorphism $F_n\to G(K)$.

The abelianization of $G(K)$ is given as 
\[
\alpha:G(K)\rightarrow G(K)/[G(K),G(K)]
\cong\Z=\langle t \rangle
\]
and the induced map on group rings as 
\[
\alpha_*:\Z G(K)\to \Z\langle t\rangle=\Z[t,t^{-1}].
\]

\begin{definition}
The $(n-1)\times n$-matrix $A$ defined by 
\[
A=\left(\alpha_*\left(\frac{\partial r_i}{\partial x_j}\right)\right)
\in M\left((n-1)\times n;\Z[t,t^{-1}]\right)
\ 
(1\leq i\leq n-1,\ 1\leq j\leq n)
\]
is called the Alexander matrix of $G(K)=\langle x_1,\ldots,x_n\ |\ r_1,\ldots,r_{n-1}\rangle$. 

\end{definition}

Let 
$A_k$ be the $(n-1)\times(n-1)$-matrix 
obtained by removing the $k$-th column from $A$.

\begin{lemma}
There exists an integer $k\in\{1,\cdots,n\}$ such that $\alpha_*(x_k)-1\neq 0\in \Z[t,t^{-1}]$.
\end{lemma}

\begin{proof}
If $\alpha(x_k)=1$ for any $k$, 
then clearly $\alpha:G(K)\rightarrow \Z$ is the trivial homomorphism, not an epimorphism. 
It contradicts that $\alpha$ is an epimorphism. 
\end{proof}

\begin{lemma}
For any $k,l\in\{1,\cdots,n\}$, 
\[
\left(\alpha_{\ast}\left(x_l\right)-1\right)
\det  A_k
=\pm 
\left(\alpha_{\ast}\left(x_k\right)-1\right)
\det  A_l.
\]
\end{lemma}

\begin{proof}
We may assume $k=1,l=2$ without the loss of generality.

For any relator $r_i=1\in \Z G(K)$, 
by applying the fundamental formula and projection on $\Z G(K)$, 
it is seen that 
\[
0=r_i-1=\sum_{j=1}^n \frac{\partial r_i}{\partial x_j}(x_j-1).
\]
By applying $\alpha_*$ to both sides, we obtain
\[
\sum_{j=1}^n \alpha_*\left(\frac{\partial r_i}{\partial x_j}\right)(\alpha_*(x_j)-1)=0.
\]
Hence one obtains  
\[
(\alpha_*(x_1)-1)\alpha_*\left(\frac{\partial r_i}{\partial x_1}\right)
=
-\sum_{j=2}^n \alpha_*\left(\frac{\partial r_i}{\partial x_j}\right)(\alpha_*(x_j)-1).
\]

Here let 
$A_2$ be the matrix removed the second column from $A$ and 
$\tilde{A}_2$ the one replaced the first column $\alpha_*\left(\frac{\partial r_i}{\partial x_1}\right)$ 
to $(\alpha_*(x_1)-1)\alpha_*\left(\frac{\partial r_i}{\partial x_1}\right)$ in $A_2$. 

Take the determinant 
\[
\begin{split}
\det \tilde{A}_2
&=\begin{vmatrix}
\left(\alpha_*\left(x_1\right)-1\right)
\alpha_*\left(\frac{\partial r_1}{\partial x_1}\right)&
\alpha_*\left(\frac{\partial r_1}{\partial x_3}\right)&\dots&
\alpha_*\left(\frac{\partial r_1}{\partial x_{n}}\right)\\
\vdots&\vdots&\dots&\vdots\\
\left(\alpha_*\left(x_1\right)-1\right)
\alpha_*\left(\frac{\partial r_{n-1}}{\partial x_1}\right)& 
\alpha_*\left(\frac{\partial r_{n-1}}{\partial x_3}\right)&\dots&
\alpha_*\left(\frac{\partial r_{n-1}}{\partial x_{n}}\right)\end{vmatrix}\\
&=
\left(\alpha_*\left(x_1\right)-1\right)
\det  A_2.
\end{split}
\]

On the other hand, 
replace 
$\left(\alpha_*\left(x_1\right)-1\right)
\alpha_*\left(\frac{\partial r_i}{\partial x_1}\right)$ 
to $\displaystyle-\sum_{j=2}^n \alpha_*\left(\frac{\partial r_i}{\partial x_j}\right)
(\alpha_*(x_j)-1)$, 
the same determinant is given by 
\[
\begin{split}
\det \tilde{A}_2
&=
\begin{vmatrix}
-\displaystyle\sum_{j=2}^n \alpha_*\left(\frac{\partial r_1}{\partial x_j}\right)(\alpha_*(x_j)-1)&\dots&
\alpha_*\left(\frac{\partial r_1}{\partial x_{n}}\right)\\
\vdots&\hdotsfor{1}&\vdots\\
-\displaystyle\sum_{j=2}^n \alpha_*\left(\frac{\partial r_{n-1}}{\partial x_j}\right)(\alpha_*(x_j)-1)&\dots&
\alpha_*\left(\frac{\partial r_{n-1}}{\partial x_{n}}\right)\\
\end{vmatrix}\\
&=
-\sum_{j=2}^n
(\alpha_*(x_j)-1)
\begin{vmatrix}
\alpha_*\left(\frac{\partial r_1}{\partial x_j}\right)&
\alpha_*\left(\frac{\partial r_1}{\partial x_3}\right)&\dots&
\alpha_*\left(\frac{\partial r_1}{\partial x_{n}}\right)\\
\vdots&\hdotsfor{2}&\vdots\\
\alpha_*\left(\frac{\partial r_{n-1}}{\partial x_j}\right)&
\alpha_*\left(\frac{\partial r_{n-1}}{\partial x_3}\right)&\dots&
\alpha_*\left(\frac{\partial r_{n-1}}{\partial x_{n}}\right)\\
\end{vmatrix}
\end{split}
\]

\[
\begin{split}
&=
-(\alpha_*(x_2)-1)
\begin{vmatrix}
\alpha_*\left(\frac{\partial r_1}{\partial x_2}\right)&
\alpha_*\left(\frac{\partial r_1}{\partial x_3}\right)&\dots&
\alpha_*\left(\frac{\partial r_1}{\partial x_{n}}\right)\\
\vdots&\hdotsfor{2}&\vdots\\
\alpha_*\left(\frac{\partial r_{n-1}}{\partial x_2}\right)&
\alpha_*\left(\frac{\partial r_{n-1}}{\partial x_3}\right)&\dots&
\alpha_*\left(\frac{\partial r_{n-1}}{\partial x_{n}}\right)\\
\end{vmatrix}\\
&=-(\alpha_*(x_2)-1)\det  A_1 .
\end{split}
\]
Therefore it holds that 
\[
\left(\alpha_*\left(x_1\right)-1\right)
\det  A_2
=-(\alpha_*(x_2)-1)\det  A_1 .
\]

\end{proof}

From these two lemmas, we can consider 
\[
\frac
{\det  A_k}
{\alpha_*(x_k)-1}
\]
as an invariant of $G(K)$ with a presentation with deficiency one. 

Now we supposed that the deficiency of a presentation is one. 
To prove this invariant is independent 
of choices of a presentation, 
up to $\pm t^s$ ($s\in\Z$), 
we define it for the case of higher deficiencies 
and apply the Tietze transformations to them. 

We take and fix a presentation of $G(K)$ as 
\[
G(K)=\langle x_1,\ldots,x_{n}\ |\ r_1,\ldots,r_{n-d}\rangle
\]
where $1\leq d\leq n-1$. 

The Alexander matrix associated to the above presentation, 
it is similarly defined by 
\[
A=\left(\alpha_*\left(\frac{\partial r_i}{\partial x_j}\right)\right)
\in M\left((n-d)\times n;\Z[t,t^{-1}]\right)\ 
(1\leq i\leq n-d,\ 1\leq j\leq n).
\]

Let 
$A_k$ be the $(n-d)\times(n-1)$-matrix 
obtained by removing the $k$-th column from $A$. 
This is not a square matrix if $d\geq 2$. 

Let $A_k^{I}$ be the $(n-d)\times(n-d)$-matrix 
consisting of the columns whose indices belong to 
$I=(i_{1},\dots,i_{n-d})\ (1\leq i_{1}<\cdots< i_{n-d}\leq n)$.

By the similar arguments for the deficiency one case, 
we can also prove the following lemma. 

\begin{lemma}
For any $k,l\in\{1,\cdots,n\}$ 
and any choice of $I$ such that $k,l\notin I$,
\[
\left(\alpha_*\left(x_l\right)-1\right)
\det  A_k^{I}
=\pm 
\left(\alpha_*\left(x_k\right)-1\right)
\det  A_l^{I}.
\]
\end{lemma}

Furthermore 
it is similarly seen 
that 
there exists an integer $k\in\{1,\cdots,n\}$ such that $\alpha_*(x_k)-1\neq 0\in \Z[t,t^{-1}]$.
Now 
we put 
$Q_{k}$ to the greatest common divisior 
of $\det A_{k}^{I}$ for all indecies $I$. 
From the above, we can consider 
\[
\frac
{Q_k}
{\alpha_*(x_k)-1}
\]
as an invariant of $G(K)$. 

\begin{remark}
For the case of $d=1$, 
we can choice the index set $I$ as $I=(1,\dots,k-1,k+1,\dots,n)$. 
Hence the above definition gives the same one 
in the case of deficiency one presentations. 
\end{remark}

Now we recall Tietze transformations as follows. 
See \cite{Magnus-Karrass-Solitar} for example. 
\begin{theorem}[Tietze]
Any presentation 
$G=\langle x_{1},\cdots,x_{k}\ |\ r_{1},\cdots,r_{l}\rangle$ can be transformed to any other presentation of $G$ 
by an application of a finite sequence of the following two type operations and their inverses:
\noindent
\flushleft{(I)}
To add a consequence $r$ of the relators 
$r_{1},\cdots,r_{l}$ 
to the set of relators. 
The resulting presentation is given by 
$\langle x_{1},\cdots,x_{k}\ |\ r_{1},\cdots,r_{l},r\rangle$. 
\flushleft{(II)}
To add a new generator $x$ and a new relator $xw^{-1}$ 
where $w$ is any word in $x_{1},\cdots,x_{k}$. 
The resulting presentation is given by 
$\langle x_{1},\cdots,x_{k},x\ |\ r_{1},\cdots,r_{l},xw^{-1}\rangle$.
\end{theorem}

We can prove the following. 

\begin{proposition}
Up to $\pm t^s$ ($s\in\Z$), 
the rational expression 
\[
\frac{Q_k}{\alpha_\ast(x_k)-1}
\] is independent of a choice of a presentation of $G(K)$.  
Namely it is an invariant of a group $G(K)$ up to $\pm t^s$ ($s\in\Z$). 
\end{proposition}

\begin{proof}
Take presentations as 
\[
P=\langle x_{1},\dots,x_{n}\ |\ r_{1},\dots,r_{n-d}\rangle
\]
and 
\[
P'=\langle x_{1},\dots,x_{n}\ |\ r_{1},\dots,r_{n-d},r\rangle
\]
by applying the Tietze transformation (I). 
Now assume $r$ has a form as
\[
r=\prod_{k=1}^{p}w_{k}r_{i_{k}}^{\epsilon_{k}}w_{k}^{-1}.
\]
where 
$1\leq i_{k}\leq n-d,w_{k}\in F_{n}$ and $\epsilon_{k}=\pm 1$ 
for $1\leq k\leq p$. 
By applying Fox's free differentials, 
one has
\[
\begin{split}
\frac{\partial r}{\partial x_{j}}
&=\sum_{k=1}^{p}
\left(
\prod_{l=1}^{k-1}w_{l}r_{i_{l}}^{\epsilon_{l}}w_{l}^{-1}
\right)
\left(
\frac{\partial w_{k}}{\partial x_{j}}
+u_{k}\frac{\partial r_{i_{k}}}{\partial x_{j}}
-w_{k}r_{i_{k}}^{\epsilon_{k}}w_{k}^{-1}\frac{\partial w_{k}}{\partial x_{j}}\right)\\
&=\sum_{k=1}^{p}
\left(
\prod_{l=1}^{k-1}w_{l}r_{i_{l}}^{\epsilon_{l}}w_{l}^{-1}
\right)
\left(
(1-w_{k}r_{i_{k}}^{\epsilon_{k}}w_{k}^{-1})
\frac{\partial w_{k}}{\partial x_{j}}
+u_{k}\frac{\partial r_{i_{k}}}{\partial x_{j}}
\right).
\end{split}
\]
Here 
\[
u_{k}=
\begin{cases}
&\ w_{k}\ (\epsilon_{k}=1),\\
&-w_{k}r_{i_{k}}^{-1}\ (\epsilon_{k}=-1).
\end{cases}
\]
Because $\alpha_{\ast}(r_{i})=1\in\Z[t,t^{-1}]$, 
one obtains
\[
\begin{split}
\alpha_{\ast}\left(\frac{\partial r}{\partial x_{j}}\right)
&=\sum_{k=1}^{p}
\alpha_{\ast}(u_{k})
\alpha_{\ast}\left(
\frac{\partial r_{i_{k}}}{\partial x_{j}}
\right)\\
&=\sum_{k=1}^{p}
\epsilon_{k}\alpha_{\ast}(w_{k})
\alpha_{\ast}\left(
\frac{\partial r_{i_{k}}}{\partial x_{j}}
\right).
\end{split}
\]
This shows the last row of the Alexander matrix $A'$
associated to $P'$ are linear combinations of $p$ rows of the Alexander matrix $A$ associated to $P$. 
It is clear that the first $n-d$ rows of $A'$
associated to $P'$ are exactly same with the first $n-d$ rows 
of $A$ associated to $P$. 
Therefore it is shown that 
the invariant 
$\displaystyle\frac{\det {A'}_{k}^{I}}{\alpha_{\ast}(x_{k})-1}$ is the same 
with the one computed by $A$.

Next take a presentation 
\[
P''
=\langle
x_{1},\dots,x_{n},x(=x_{n+1})\ |\ r_{1},\dots,r_{n-d},xw^{-1}
\rangle
\]
obtained from $P$ by applying the Tietze transformation (II). 
By direct computations, 
we see the Alexander matrix $A''$ associated to $P''$ has the form of 
\[
A''=
\begin{pmatrix}
A & 0\\
\ast & 1
\end{pmatrix}
\]
where 
the last row is 
\[
\begin{split}
&\left(
-\alpha_{\ast}(x)\alpha_\ast(w)\alpha_{\ast}\left(\frac{\partial w}{\partial x_{1}}\right), 
\dots,
-\alpha_{\ast}(x)\alpha_\ast(w)\alpha_{\ast}\left(\frac{\partial w}{\partial x_{n}}\right), 1
\right)\\
=&
\left(
-\alpha_{\ast}(x)\alpha_\ast(w)\alpha_{\ast}\left(\frac{\partial x_{n+1}}{\partial x_{1}}\right), 
\dots,
-\alpha_{\ast}(x)\alpha_\ast(w)\alpha_{\ast}\left(\frac{\partial x_{n+1}}{\partial x_{n}}\right), 1
\right)\\
=&\left(
0,\dots,0, 1
\right).
\end{split}
\]
Here suppose $\alpha_{\ast}(x_{k})-1\neq 0$. 
Then 
the determinant of ${A''}_{k}^{J}$ 
for an index set $J=(j_{1},\dots,j_{n-d+1})$ 
can be non-zero 
if and only if $J$ has the form 
$J = (j_{1},\dots,j_{n-d},n+1)$. 
Then for $J = (j_{1},\dots,j_{n-d},n+1)$ and $I= (j_{1},\dots,j_{n-d})$, 
it is seen 
\[
\det {A''}_{k}^{J}=\det A_{k}^{I}.
\]
Hence it holds  
\[
\frac{Q_k({A''})}{\alpha(x_{k})-1}
=\frac{Q_k(A)}{\alpha(x_{k})-1}.
\]
This completes the proof.
\end{proof}

For any knot $K$, we can take some special presentation of $G(K)$, 
which is a Wirtinger presentation derived from a regular diagram on the plane. 
In this case we may assume $\alpha(x_1)=\cdots =\alpha(x_{n})=t$. 
Hence the numerator is always $t-1$. 
Therefore the numerator itself is an invariant of $G(K)$ up to $\pm t^s$. 

\begin{definition}
This is called the Alexander polynomial $\Delta_{K}(t)=\det A_{k}$ of $K$.
\end{definition}

\begin{remark}
It is clear that Alexander polynomial is well-defined up to $\pm t^{s}$.
\end{remark}
%%%%%%%%%%%%%%%%%%%%
\subsection{Examples}

\begin{example}
We consider the trefoil knot $3_1=T(2,3)$ first. 

%\begin{figure}[htbp]
%\centering
%\includegraphics[width=6cm]{diagram_31.eps}
%{$3_{1}=T(2,3)$}
%\end{figure}

Fix the following presentation 
\[
G(3_1)=\langle x,y\ |\ r=xyx(yxy)^{-1}\rangle. 
\]

By applying the abelianization $\alpha$, 
the relator $r=xyx(yxy)^{-1}$ goes to 
\[
\begin{split}
\alpha(r)
&=
\alpha(x)\alpha(y)\alpha(x)\alpha(y)^{-1}\alpha(x)^{-1}\alpha(y)^{-1}\\
&=\alpha(x)\alpha(y)^{-1}\in G(3_1)/[G(3_1),G(3_1)].
\end{split}
\]
Because $\alpha(r)=1$, 
then we get 
\[
\alpha(x)\alpha(y)^{-1}=1
\in G(3_1)/[G(3_1),G(3_1)]. 
\]

Hence the abelianization can be given by 
\[
\alpha:G(3_1)\ni x,y\mapsto t\in \langle t\rangle.
\]
By applying $\displaystyle\frac{\partial}{\partial x}$ to $r$ 
and mapping it on $\Z G(3_{1})$, 
we have 
\[
\begin{split}
\frac{\partial}{\partial x}(r)
&=\frac{\partial}{\partial x}(xyx(yxy)^{-1})\\
&=\frac{\partial}{\partial x}(xyx)-xyx(yxy)^{-1}\frac{\partial}{\partial x}(yxy)\\
&=\frac{\partial}{\partial x}(xyx)-r\frac{\partial}{\partial x}(yxy)\\
&=
\frac{\partial}{\partial x}\left(xyx\right)-\frac{\partial}{\partial x}\left(yxy\right)\\
&=
\frac{\partial}{\partial x}\left(xyx-yxy\right).
\end{split}
\]
Here we used the property $r=1$ in $\Z G(3_1)$. 
Therefore we can compute free differentials for $xyx-yxy$
instead of $r=xyx(yxy)^{-1}$. 

Accordingly we compute 
\[
\begin{split}
\frac{\partial}{\partial x}(xyx-yxy)
&=\frac{\partial}{\partial x}(xyx)-\frac{\partial}{\partial x}(yxy)\\
&=1+xy-y\\
&\overset{\alpha_*}{\mapsto}
t^2-t+1\in\Z[t,t^{-1}].
\end{split}
\]

Similarly
\[
\begin{split}
\frac{\partial}{\partial y}(xyx-yxy)
&=\frac{\partial}{\partial y}(xyx)-\frac{\partial}{\partial y}{yxy}\\
&=x-1-yx\\
&\overset{\alpha_*}{\mapsto}
-(t^2-t+1)\in\Z[t,t^{-1}].
\end{split}
\]
Hence one has
\[
A=\begin{pmatrix}
(t^2-t+1) & -(t^2-t+1)
\end{pmatrix},
\]
and 
\[
\begin{split}
\frac{\text{det} A_2}{t-1}
&=-\frac{\text{det} A_1}{t-1}\\
&=\frac{t^2-t+1}{t-1}.
\end{split}
\]

By changing this presentation to 
$\langle x,y,z\ |\ xyx(yxy)^{-1}, xyz^{-1}\rangle$, 
then Alexander matrix is changed to 
\[
A=\begin{pmatrix}
(t^2-t+1) & -(t^2-t+1) & 0\\
1 & t & -1
\end{pmatrix}.
\]
In this case the abelianization $\alpha$ is given by 
\[
\alpha(x)=\alpha(y)=t, \alpha(z)=t^{2}. 
\]
From this Alexander matrix, 
we obtain   
\[
\begin{split}
\frac{\text{det} A_1}{t-1}&=\frac{t^2-t+1}{t-1}, \\
\frac{\text{det} A_2}{t-1}&=-\frac{t^2-t+1}{t-1}, \\
\frac{\text{det} A_3}{t^2-1}&=\frac{t(t^2-t+1)+(t^2-t+1)}{t^2-1}=\frac{t^2-t+1}{t-1}.
\end{split}
\]
Therefore the Alexander polynomial of the trefoil knot 
is given by 
\[
\Delta_{3_{1}}(t)=t^{2}-t+1\]
up to $\pm t^{s}$.
\end{example}

%%%%%%%%%%%%%%%%%%%%%%%%%%%%%%
\begin{example}{Figure-eight knot $4_1$}

%\begin{figure}[htbp]
%\centering
%\includegraphics[width=6cm]{diagram_41.eps}
%{$4_{1}$}
%\end{figure}

Take a presentation of $G(4_1)$ as 
\[
G(4_{1})=
\langle
x,y\ |\ wxw^{-1}=y\rangle
\]
where $w=x^{-1}yxy^{-1}$.

Using this presentation, 
the abelianization 
$\alpha:G(4_1)\rightarrow \langle t \rangle$ 
is given by $\alpha(x)=\alpha(y)=t$. 

Then one has 
\[
\begin{split}
\frac{\partial}{\partial x}(wxw^{-1}y^{-1})
&=\frac{\partial w}{\partial x}+w\frac{\partial x}{\partial x}-wxw^{-1}\frac{\partial w}{\partial x}\\
&=(1-y)\frac{\partial w}{\partial x}+w\\
&\overset{\alpha_*}{\mapsto}
(1-t)\alpha_\ast\left(\frac{\partial w}{\partial x}\right)+1
\end{split}
\]
and 
\[
\begin{split}
\alpha_*\left(\frac{\partial w}{\partial x}\right)
&=\alpha_*\left(\frac{\partial}{\partial x}(x^{-1}yxy^{-1})\right)\\
&=\alpha_*(-x^{-1}+x^{-1}y)\\
&=-t^{-1}+1.
\end{split}
\]

Consequently it is seen that 
\[
\begin{split}
\alpha_*\left(\frac{\partial}{\partial x}(wxw^{-1}y^{-1})\right)
&=(1-t)(-t^{-1}+1)+1\\
&=-t^{-1}+1+1-t(-t^{-1}+1)\\
&=-t^{-1}+1+1+1-t\\
&=-t^{-1}+3-t.
\end{split}
\]

Similarly one has
\[
\begin{split}
\alpha_*\left(\frac{\partial}{\partial y}(wxw^{-1}y^{-1})\right)
&=\alpha_*\left((1-y)\frac{\partial w}{\partial x}-1\right)\\
&=(1-t)(t^{-1}-1)-1\\
&=t^{-1}-3+t.\\
\end{split}
\]
Hence we obtain  
\[A=\begin{pmatrix}
-t^{-1}+3-t & t^{-1}-3+t
\end{pmatrix}
\]
and 
\[
\begin{split}
\frac{\text{det} A_1}{\alpha_*(x_1)-1}
&=\frac{t^{-1}-3+t}{t-1}\\
&=-\frac{1}{t}\frac{(-t^2+3t-1)}{t-1},
\end{split}
\]

\[
\begin{split}
\frac{\text{det} A_2}{\alpha_*(x_2)-1}
&=-\frac{t^{-1}-3+t}{t-1}\\
&=\frac{1}{t}\frac{(-t^2+3t-1)}{t-1}.
\end{split}
\]
Finally, the Alexander polynomial of the figure-eight knot is given by  
\[
\Delta_{4_{1}}(t)=-t^2+3t-1
\]
up to $\pm t^{s}$.
\end{example}

%%%%%%%%%%%%%%%%%%%%
\section{Reidemeister torsion}

In this section we explain the theory of the Reidemeister torsion, 
which is an invariant 
of a compact CW-complex with a linear representation of the fundamental group. 

Let $K$ be a knot in $S^3$ 
and $G(K)$ the knot group of $K$. 
We take an open tubular neighborhood $N(K)\subset S^3$ of $K$ 
and the exterior $E(K)=S^3\setminus N(K)$ of $K$. 
The knot exterior $E(K)$ is a compact 3-manifold with a torus boundary. 
Note that $\pi_{1}(E(K))$ is isomorphic to $G(K)$ by natural inclusion $E(K)\rightarrow S^{3}\setminus K$. 

Here we consider the abelianization 
$\alpha:G(K)\rightarrow T=\langle t\rangle\subset \mathit{GL}(1;\Q(t))$ 
as an 1-dimensional representation over $\Q(t)$.
Here $\Q(t)$ denotes the one variable rational function field over $\Q$. 
Now we can define Reidemeister torsion 
\[
\tau_\alpha(E(K))\in\Q(t)
\] 
of $E(K)$ for $\alpha$. 
We mention the following well-known theorem by Milnor\cite{Milnor62} 
before giving the definition of Reidemeister torsion. 

\begin{theorem}[Milnor]
\[
\frac{\Delta_{K}(t)}{t-1}=\tau_{\alpha}(E(K)). 
\]
\end{theorem}

\begin{remark}
Both of left and right hand sides are well defined up to $\pm t^{s}$.
\end{remark}

\subsection{Algebraic definitions}

Recall the definition of Reidemeister torsion. 

Let $C_*$ be a chain complex over a field $\F$ as 
\[
0\longrightarrow C_m
\overset{\partial_m}{\longrightarrow} C_{m-1}
\overset{\partial_{m-1}}{\longrightarrow} C_{m-2}
\longrightarrow \dots
\overset{\partial_2}{\longrightarrow} C_1
\overset{\partial_1}{\longrightarrow} C_0
\longrightarrow 0.
\]
Because 
$0\longrightarrow 
Z_q(=\textit{Ker}\partial_q)
\longrightarrow 
C_q
\overset{\partial_{{q}}}{\longrightarrow}
B_{q-1}(=\textit{Im}\partial_{{q}})
\longrightarrow 0$ is exact, then we have an isomorphism 
\[
C_q\cong Z_q\oplus B_{q-1},
\] 
which is not canonical. 
Note that a pair of bases of $Z_q$ and $B_{q-1}$ 
gives a basis of $C_q$. 

%%%%%%%%%%%%
\begin{definition}
A chain complex $C_*$ is called to be acyclic if 
$B_q=Z_q$ 
for $q=0,1,\cdots,m$, 
that is, all homology groups $H_\ast(C_\ast)=0$ .
\end{definition}

From here we assume $C_\ast$ is acyclic and  
further a basis $\textbf{c}_q$ of $C_q$ is given for any $q$. 
That is, $C_\ast$ is a based acyclic chain complex 
of finite dimensional vector spaces over $\F$.

Here take a basis $\textbf{b}_q$ on $B_q$ for any $q$. 

On the above exact sequence 
\[
0\longrightarrow 
Z_q
\longrightarrow 
C_q
\overset{\partial_{{q}}}{\longrightarrow}
B_{q-1}
\longrightarrow 0,
\]
take a lift $\tilde{\textbf{b}}_{q-1}$ of $\textbf{b}_{q-1}$. 
Now a pair $(\textbf{b}_q,\tilde{\textbf{b}}_{q-1})$ gives a basis on $C_q$. 
Here two basis $\textbf{c}_q$ and $(\textbf{b}_q,\tilde{\textbf{b}}_{q-1})$ gives an isomorphism 
\[
C_q\cong B_q\oplus B_{q-1}.
\]

For any two bases 
$\textbf{b}=\{b_1,\cdots,b_n\}, \textbf{c}=\{c_1,\cdots,c_n\}$ of a vector space $V$ over $\F$. 
then there exists a non-singular matrix $P=(p_{ij})\in\GL(n;\F)$ 
such that $\displaystyle b_j=\sum_{i=1}^n p_{ji}c_i$. 

\begin{definition}
$P$ is called the transformation matrix from $\textbf{c}$ to $\textbf{b}$.
\end{definition}

Under this definition, 
we simply write  
$\left(\textbf{b}_q,\tilde{\textbf{b}}_{q-1}/\textbf{c}_q\right)$ 
for the transformation matrix 
from $\textbf{c}_q$ to $(\textbf{b}_q,\tilde{\textbf{b}}_{q-1})$
and 
$\left[\textbf{b}_q,\tilde{\textbf{b}}_{q-1}/\textbf{c}_q\right]$ 
for the determinant $\det \left(\textbf{b}_q,\tilde{\textbf{b}}_{q-1}/\textbf{c}_q\right)$. 

\begin{lemma}
The determinant $[\textbf{b}_q,\tilde{\textbf{b}}_{q-1}/\textbf{c}_q]$ 
is independent on choices of a lift $\tilde{\textbf{b}}_{q-1}$. 
Hence we can simply write $[\textbf{b}_q,\textbf{b}_{q-1}/\textbf{c}_q]$ to it.
\end{lemma}

\begin{proof}
Take another lift $\hat{\textbf{b}}_{q-1}$ of $\textbf{b}_{q-1}$ on $C_q$. 
For example, one vector $v$ in $\tilde{\textbf{b}}_{q-1}$ 
is replaced 
to another vector $v'$ in $\hat{\textbf{b}}_{q-1}$. 
But $v, v'$ map to the same vector in $B_{q-1}$. 
Here 
\[
0\longrightarrow 
Z_q\longrightarrow C_q
\longrightarrow B_{q-1}
\longrightarrow 0
\]
is an exact sequence, 
then a difference $v-v'$ belongs to $Z_q=B_q$. 
Hence $v-v'$ can be expressed as a linear combination of the vectors of $\textbf{b}_{q}$. 
Then by the definition of the determinant, it can be seen that 
\[
\left[\textbf{b}_q,\tilde{\textbf{b}}_{q-1}/\textbf{c}_q\right]
=
\left[\textbf{b}_q,\hat{\textbf{b}}_{q-1}/\textbf{c}_q\right]. 
\]

Therefore the determinant is not changed. 
\end{proof}

\begin{definition}
The torsion $\tau(C_*)$ of a based chain complex $(C_*,\{\textbf{c}_{q}\})$ is defined by 
\[
\tau(C_*)=
\frac{\prod_{q:odd}[\textbf{b}_q,\textbf{b}_{q-1}/\textbf{c}_q]}
{\prod_{q:even}[\textbf{b}_q,\textbf{b}_{q-1}/\textbf{c}_q]}
\in \F\setminus\{0\}.
\]
\end{definition}

\begin{lemma}
The torsion $\tau(C_*)$ is independent of choices of $\textbf{b}_0,\cdots,\textbf{b}_m$. 
\end{lemma}

\begin{proof}
Assume $\textbf{b}'_q$ is another basis of $B_q$. 

In the definition of $\tau(C_*)$, 
the difference between $\textbf{b}_q$ and $\textbf{b}'_q$ is related 
to the following only two parts $\left[\textbf{b}'_q,\textbf{b}_{q-1}/\textbf{c}_{q}\right]$ 
and $\left[\textbf{b}_{q+1},\textbf{b}'_q/\textbf{c}_{q+1}\right]$. 
By standard arguments of the linear algebra, 
\[
\left[\textbf{b}'_q,\textbf{b}_{q-1}/\textbf{c}_{q}\right]
=
\left[\textbf{b}_q,\textbf{b}_{q-1}/\textbf{c}_q\right]
\left[\textbf{b}'_q/\textbf{b}_q\right],
\]
\[
\left[\textbf{b}_{q+1},\textbf{b}'_q/\textbf{c}_{q+1}\right]
=\left[\textbf{b}_{q+1},\textbf{b}_q/\textbf{c}_{q+1}\right]
\left[\textbf{b}'_q/\textbf{b}_q\right].
\]
Since $\left[\textbf{b}'_q/\textbf{b}_q\right]$ appears 
in the both of the denominator and the numerator of the definition, 
they can be cancelled. 
\end{proof}

\begin{example}
Put $m=4$. Now consider 
\[
C_{\ast}:0\rightarrow C_{4}\rightarrow C_{3}\rightarrow C_{2}\rightarrow C_{1}\rightarrow C_{0}\rightarrow 0.
\]
As $\textbf{b}_4$ and $\textbf{b}_{-1}$ are zero, then by the definition, 
one has
\[
\begin{split}
\tau(C_*)
&=
\frac
{[\textbf{b}_4,\textbf{b}_{3}/\textbf{c}_4]
[\textbf{b}_2,\textbf{b}_{1}/\textbf{c}_2]
[\textbf{b}_0,\textbf{b}_{-1}/\textbf{c}_0]}
{[\textbf{b}_3,\textbf{b}_{2}/\textbf{c}_3]
[\textbf{b}_1,\textbf{b}_{0}/\textbf{c}_1]}\\
&=
\frac
{[\textbf{b}_{3}/\textbf{c}_4]
[\textbf{b}_2,\textbf{b}_{1}/\textbf{c}_2]
[\textbf{b}_0/\textbf{c}_0]}
{[\textbf{b}_3,\textbf{b}_{2}/\textbf{c}_3]
[\textbf{b}_1,\textbf{b}_{0}/\textbf{c}_1]}.
\end{split}
\]
In this case, 
the number of factors in the denominator 
and the number of factors in the numerator are not same. 
However it can be seen that $\tau(C_{\ast})$ is independent 
of choices of $\textbf{b}_{0},\textbf{b}_{1},\textbf{b}_{2},\textbf{b}_{3}$. 
\end{example}

\begin{example}
Next we put $m=3$. Here 
\[
C_{\ast}:0\rightarrow C_{3}\rightarrow C_{2}\rightarrow C_{1}\rightarrow C_{0}\rightarrow 0.
\]
As $\textbf{b}_3$ and $\textbf{b}_{-1}$ are zero, 
then one has 
\[
\begin{split}
\tau(C_*)
&=
\frac
{
[\textbf{b}_2,\textbf{b}_{1}/\textbf{c}_2]
[\textbf{b}_0,\textbf{b}_{-1}/\textbf{c}_0]}
{[\textbf{b}_3,\textbf{b}_{2}/\textbf{c}_3]
[\textbf{b}_1,\textbf{b}_{0}/\textbf{c}_1]}\\
&=
\frac
{
[\textbf{b}_2,\textbf{b}_{1}/\textbf{c}_2]
[\textbf{b}_0/\textbf{c}_0]}
{[\textbf{b}_{2}/\textbf{c}_3]
[\textbf{b}_1,\textbf{b}_{0}/\textbf{c}_1]}.
\end{split}
\]

In this case the numbers of factors are same. 
Similarly it can be seen that $\tau(C_{\ast})$ is independent 
of choices of $\textbf{b}_{0},\textbf{b}_{1},\textbf{b}_{2},\textbf{b}_{3}$
\end{example}
%%%%%%%%%%%%%
The following lemma is well-known as Mayer-Vietoris argument 
for a torsion invariant. 
See \cite{Milnor66} for the proof. 

\begin{lemma}
Let  
$0\rightarrow C'_*\rightarrow C_*\rightarrow C''_*\rightarrow 0$ 
be a short exact sequence of based chain complexes. 
Assume  the bases of $C_{*}$ are given as pairs of $(\textbf{c}_*',\textbf{c}_*'')$ 
where $\{\textbf{c}_*'\},\{\textbf{c}_*''\}$ are bases of $C_{*}', C_{*}''$. 
If two of $C'_*,C_*,C''_*$ are acyclic, 
then the third one is also acyclic and 
\[
\tau(C_*)=\pm \tau(C'_*)\tau(C''_*).
\]
\end{lemma}

\begin{remark}
The reason why the signs $\pm $ appear in the right hand side is the following. 
To define the torsions we use the following isomorphisms;
\begin{itemize}
\item
$C'_{\ast}\cong Z'_{\ast}\oplus B'_{\ast}$, $C_{\ast}\cong Z_{\ast}\oplus B_{\ast}$, $C''_{\ast}\cong Z''_{\ast}\oplus B''_{\ast}$.
\end{itemize}
On the other hand, to get this formula, we use 
\begin{itemize}
\item
$C_{\ast}\cong C'_{\ast}\oplus C''_{\ast}\cong Z'_{\ast}\oplus B'_{\ast}\oplus Z''_{\ast}\oplus B''_{\ast}$.
\end{itemize}
Here the signs appear as we need to change orders of vectors in general. 
\end{remark}

%%%%%%%%%%%
\subsection{Geometric settings}

Now we apply this torsion invariant of chain complexes to the following geometric situations. 

Let $X$ be a finite CW-complex and $\tilde X$ a universal covering of $X$. 
We lift a CW-complex structure of $X$ on $\tilde X$. 
The fundamental group $\pi_1 X$ acts on $\tilde X$ from the right-hand side as deck transformations. 
By applying the cellular approximation theorem, 
we may assume this action is free and cellular 
under taking subdivisions if it is needed. 
Then the chain complex $C_*(\tilde{X};\Z)$ has the structure of a chain complex of free $\Z[\pi_1 X]$-modules. 

Let $\rho:\pi_1 X\rightarrow \GL(V)$  be an $n$-dimensional linear representation 
over a field $\F$. 
Using the representation $\rho$, 
$V$ admits a structure of a $\Z[\pi_1 X]$-module, 
which is denoted by $V_\rho$. 
Define the chain complex $C_*(X; V_\rho)$ 
by $C_*({\tilde X}; \Z)\otimes_{\Z[\pi_1 X]} V_\rho$. 
Here we choose a preferred basis of $C_i(X; V_\rho)$ for any $i$ 
as 
\[
(\tilde{u}_1\otimes \textbf{e}_1, \dots,\tilde{u}_1\otimes \textbf{e}_n, \dots,\tilde{u}_d\otimes\textbf{e}_1, \dots,\tilde{u}_d\otimes\textbf{e}_n)
\]
where $\{\textbf{e}_1 ,\dots,\textbf{e}_n\}$ is a basis of $V$, 
$\{u_1,\dots,u_d \}$ are the $i$-cells 
giving a basis of $C_i(X; \Z)$ and 
$\{\tilde{u}_1,\dots,\tilde{u}_d\}$ are lifts of them in $C_i(\tilde{X}; \Z)$. 

Now we suppose that $C_*(X; V_\rho)$ is acyclic, 
namely all homology groups $H_*(X; V_\rho)$ are vanishing. 
In this case we call $\rho$ an acyclic representation. 

\begin{definition}
Reidemeister torsion of $X$ for a representation $\rho$ 
is defined by 
\[
\tau_\rho(X)=\tau(C_*(X;V_\rho))\in\F\setminus \{0\}.
\]
\end{definition}

\begin{remark}
Reidemeister torsion $\tau_\rho(X)$ does not depend on the choices 
up to $\pm f$ 
where $f\in\mathit{Im}\{\det\circ\rho:\pi_{1}(X)\rightarrow\F\setminus\{0\}\}$.
See \cite{Milnor66} for the proof. 
\end{remark}

We apply the Reidemeister torsion 
for a knot $K$ in $S^{3}$ as follows. 
Fix a CW-complex structure on $E(K)$. 
We take its universal cover 
\[
\tilde E(K)\rightarrow E(K)
\]
and also a lift of the CW-complex structure of $E(K)$ to $\tilde E(K)$. 
By applying the cellular approximation theorem, 
we may assume $G(K)$ acts freely and cellularly on $\tilde E(K)$ 
from the right 
as deck transformations.

Now we can consider 
the abelianization 
$\alpha:G(K)\rightarrow \langle t\rangle\subset \GL(1;\Q(t))$ 
as a $1$-dimensional representation of $G(K)$ over the rational function field $\Q(t)$. 

Hence the chain complex of $E(K)$ 
with $\Q(t)_\alpha$-coefficients is defined by 
\[
C_*(E(K);\Q(t)_\alpha)=C_*(\tilde E(K);\Z)\otimes_{\Z G(K)} \Q(t)_\alpha.
\]
Here we take bases 
$\textbf{c}_i$ for $C_i(E(K);\Q(t)_\alpha)$ as 
\[
(\tilde{u}_1\otimes \textbf{1}, \dots,\tilde{u}_d\otimes\textbf{1})
\]
by using lifts of $i$-cells 
$\{u_1,\dots,u_d\}$ in $E(K)$ 
and a basis $\bf 1$ for the 1-dimensional vector space $\Q(t)$ over itself 
as we explained. 

Reidemeister torsion of $E(K)$ can be defined 
\[
\tau_\alpha(E(K))=\tau(C_*(E(K);\Q(t)_\alpha))\in \Q(t)\setminus\{0\}
\]
up to $\pm t^s$. 

From Milnor's theorem, some properties of Reidemeister torsion  induce properties of Alexander polynomial. 
For example, recall one of well known properties, 
which was proved by Seifert first. 
This can proved by using properties of Reidemeister torsion. 

\begin{theorem}[Seifert\cite{Seifert}, Milnor\cite{Milnor62}]
For any knot $K$, it holds
\[
\Delta_K(t^{-1})=\Delta_K(t)
\] 
up to $\pm t^{s}$.\end{theorem}

We also have the following fact on the Alexander polynomial for a slice knot, 
which can be proved from the property of Reidemeister torsion. 
A slice knot is defined as follows. 
Now we consider $S^{3}=\partial B^{4}$. 

\begin{definition}
A knot $K\subset S^{3}$ is called a slice knot if there exists an embedded disk $D\subset B^{4}$ 
such that $\partial D=K\subset S^{3}=\partial B^{4}$.
\end{definition}

The next theorem is well-known and  classical theorem. 
It can be proved by using Reidemeister torsion. 

\begin{theorem}[Fox-Milnor\cite{Fox-Milnor}]
If $K$ is a slice knot, 
then the Alexander polynomial $\Delta_K(t)$ has a form of $\Delta_K(t)=\pm t^{s}f(t)f(t^{-1})$ 
where $f(t)\in\Z[t]$. 
\end{theorem}

%%%%%%%%%%%%%%%%%%%%%%%%%%%%%%%%%%%%%%%%%%%%%%%%%%%%%%%%%%%%%%%%%
\section{Order and obstruction}

Here we would like to mention two more things related with the Alexander polynomial;
\begin{itemize}
\item
an order of $H_{1}(E(K);\Q[t,t^{-1}]_{\alpha})$.
\item
an obstruction to deform an abelian representation.
\end{itemize}

It is seen $
H_{1}(E(K);\Q[t,t^{-1}]_{\alpha})\cong
H_{1}(E(K)_{\infty};\Q)\text{ as a }\Q[t,t^{-1}]-\text{module}$ 
where 
$E(K)_{\infty}\rightarrow E(K)$ is the $\Z$-covering 
corresponding to the abelianization epimorphism 
$\alpha:G(K)=\pi_{1}(E(K))\rightarrow\Z=\langle t\rangle$. 

An order of a finitely generated module over a principal ideal domain 
is defined as follows. 
This is a generalization of an order of an abelian group.

Let $M$ be a finitely generated $\Q[t,t^{-1}]-$module without free parts. 
From the structure theorem of a finitely generated module over a principal ideal domain, 
one has 
\[
M \cong\Q[t,t^{-1}]/(p_1)\oplus \cdots \oplus\Q[t,t^{-1}]/(p_k)
\]
where 
$p_1,\cdots,p_k\in \Q[t,t^{-1}]$ 
such that 
\[\Q[t,t^{-1}]\supsetneq(p_{1})\supset (p_{2})\supset \cdots \supset(p_{k})\neq(0).
\]

\begin{definition}
The order ideal $\text{ord}(M)$ of $M$ is defined by 
\[
\text{ord}(M)=(p_1\cdots p_k)\subset\Q[t,t^{-1}].
\]
\end{definition}

Applying an order to the case of $H_\ast(E(K);\Q[t,t^{-1}]_\alpha)$, 
the following proposition holds. 

\begin{proposition}
\noindent
\begin{itemize}
\item
$\text{ord}(H_{1}(E(K);\Q[t,t^{-1}]_{\alpha}))
=(\Delta_{K}(t))$.
\item
$\text{ord}(H_{0}(E(K);\Q[t,t^{-1}]_{\alpha}))
=(t-1)$.
\end{itemize}
\end{proposition} 

Put \cite{Milnor67} as a reference.

Next we mention that 
the Alexander polynomial is an obstruction 
to deform an 1-dimensional abelian representation 
\[
\alpha_a:G(K)\rightarrow \C^\ast=\C\setminus\{0\}
\subset \C\rtimes \C^{\ast}\subset \GL(2;\C).
\] 

Take 
$G(K)=\langle x_{1},\cdots,x_{n}\ |\ r_{1},\cdots,r_{n-1}\rangle$
 be a Wirtinger presentation of $G(K)$. 
By putting $t=a\neq 0$, 
one has a 1-dimensional abelian representation  
\[
\alpha_{a}=\alpha|_{t=a}:G(K)\ni x_{i}\mapsto a\in \C. 
\]
We put 
$\rho_{a}(x_{i})
=\begin{pmatrix} 
a & b_{i}\\ 
0 & 1
\end{pmatrix}
\in\GL(2;\C)$ 
for the image of $x_i$. 
Now a map 
\[
\rho_a:\{x_1,\dots,x_n\}\rightarrow \GL(2;\C)
\]
is given. 
If all $b_{1},\cdots,b_{n}=0$, 
then clearly $\rho_{a}$ gives a representation 
\[
\rho_{a}:G(K)\ni x_{i}
\mapsto 
\begin{pmatrix} a & 0\\0 & 1\end{pmatrix}\in \GL(2;\C).
\]
However it is also an abelian representation. 
Assume $b_i\neq 0$ for some $i$. 
Here we consider the following problem. 

\begin{problem}
When $\rho_{a}$ can be extended as a non abelian representation ?
\end{problem}

The answer is given by the next theorem. 

\begin{theorem}[de Rham\cite{deRham}]
A map $\rho_{a}$ gives a representation if and only if $\Delta_{K}(a)=0$.
\end{theorem}

\begin{remark}
This is one motivation for Wada to define twisted Alexander polynomial, 
which is how we can generalize an obstruction for a higher dimensional representation. 
\end{remark}

%%%%%%%%%%%%
\section{Twisted Alexander polynomial}

Historically there are two studies first 
to give a generalization of the Alexander polynomial 
by Lin\cite{Lin} and Wada\cite{Wada}. 
In this paper we follow the definition due to Wada, because 
it is most computable by using free differentials 
and 
it can be related to Reidemeister torsion of $E(K)$ directly. 

Recall $K$ is a knot in $S^{3}$ and $G(K)$ is the knot group. 
For simplicity we consider a representation of $G(K)$ 
in a 2-dimensional unimodular group over a field $\F$. 
From this assumption the twisted Alexander polynomial is well-defined 
up to $t^{2s}(s\in\Z)$ 

\begin{remark}
Wada defined the twisted Alexander polynomial 
for any finite presentable group 
with an epimorphism onto a free abelian group 
and a $\GL(l;R)$-representation 
over a Euclidean domain $R$. 
\end{remark}

Fix a presentation as 
\[
G(K)=\langle x_1,\ldots,x_n~|~r_1,\ldots,r_{n-1}\rangle
\]
with deficiency one. 
Let $\rho:G(K)\to \SL(2;\F)$ be a representation.
Let $M(2;\F)$ be the matrix algebra 
of $2\times 2$ matrices over $\F$. 
We write 
\[
\rho_\ast:\Z G(K) \rightarrow \Z\SL(2;\F)\cong M(2;\F)
\]
for a ring homomorphism induced by $\rho$ 
and 
\[
\alpha_\ast:\Z G(K)\rightarrow \Z\Z=\Z\langle t\rangle\cong\Z[t,t^{-1}]
\] 
for a ring homomorphism induced by $\alpha$. 
By taking the tensor product of them, 
we obtain an induced ring homomorphism
\[
\rho_\ast\otimes\alpha_\ast:
\Z G(K)\to M(2;\F)\otimes\Z[t,t^{-1}]\cong M\left(2;\F[t,t^{-1}]\right)
\]
and 
\[
\Phi:\Z F_n\to M\left(2;\F[t,t^{-1}]\right)
\] 
the composite of 
$\Z F_n\to\Z G(K)$ 
induced by the presentation 
and 
\[
\rho_\ast\otimes\alpha_\ast:\Z G(K)\rightarrow M\left(2;\F[t,t^{-1}]\right). 
\]

\begin{definition}
The $(n-1)\times n$ matrix $A_{\rho}$ 
whose $(i,j)$ component is the $2\times 2$ matrix 
$$
\Phi\left(\frac{\partial r_i}{\partial x_j}\right)
\in M\left(2;\F[t,t^{-1}]\right),
$$ 
this matrix is called 
the {twisted} Alexander matrix of a knot group 
$G(K)=\langle x_1,\ldots,x_n~|~r_1,\ldots,r_{n-1}\rangle$ associated to $\rho$. 
\end{definition}

\begin{remark}
This matrix $A_{\rho}$ can be considered as 
\[
\begin{split}
A_{\rho}
&\in M\left((n-1)\times n;M\left(2;\F[t,t^{-1}]\right)\right)\\
&=M\left(2(n-1)\times 2n;\F[t,t^{-1}]\right).
\end{split}
\]
\end{remark}

Let $A_{\rho,k}$ be the $(n-1)\times(n-1)$ matrix obtained from $A_{\rho}$ 
by removing the $k$-th column.
Then one has 
\[
\begin{split}
A_{\rho,k}
&\in M\left((n-1)\times (n-1);M\left(2;\F[t,t^{-1}]\right)\right)\\
&=M\left(2(n-1)\times 2(n-1);\F[t,t^{-1}]\right).
\end{split}
\]

By similar arguments for Alexander polynomials, 
the following two lemmas can be seen. 

\begin{lemma}
There exists $k$ such that 
$\det\Phi(x_k-1)\not=0$.
\end{lemma}

\begin{lemma}
$\displaystyle
({\det A_{\rho,k}}) ({\det\Phi(x_j-1)})
=
({\det A_{\rho,j}}) ({\det\Phi(x_k-1)})
$ for any $j,k$. 
\end{lemma}

\begin{remark}
The signs $\pm$ do not appear in the case of even dimensional unimodular representations. 
\end{remark}

From 
the above two lemmas, 
we can define the {twisted Alexander polynomial} of 
$G(K)$ associated $\rho:G(K)\to \SL(2;\F)$ to be a {rational expression} as follows. 
\begin{definition}
The twisted Alexander polynomial of $K$ for $\rho$ is defined by 
\[
\Delta_{K,\rho}(t)
=\frac{\det A_{\rho,k}}{\det \Phi(x_k-1)} 
\]
for any $k$ such that $\det\Phi(x_k-1)\not=0$. 
\end{definition}

This gives an invariant of $K$ with $\rho$. 
The following proposition can be proved by using similar arguments in the case of the Alexander polynomial. 

\begin{proposition}
\item
Up to $ct^{2s}\ (c\in\F, s\in\Z)$, 
$\Delta_{K,\rho}(t)$ is an invariant of $(G(K),\rho)$. 
Namely, it does not depend on choices of a presentation. 
\end{proposition}

Now we assume that we always take a Wirtinger presentation of $G(K)$. 
Hence we assume the deficiency is always one. 
In this case one has the more strict invariant as follows. 
However the deficiency is changed by the Tietze transformation $(\mathrm{I})$.

Now we introduce the strong Tietze transformations for a presentation of a group. 

\noindent
$\mathrm{(I_a)}$:
Replace a relator $r_i$ by its inverse $r_i^{-1}$.

\noindent
$\mathrm{(I_b)}$:
Replace a relator $r_i$ by its conjugate $w r_i w^{-1}$.

\noindent
$\mathrm{(I_c)}$:
Replace a relator $r_i$ by $r_i r_k(i\neq k)$.

\begin{remark}
The deficiency is not changed by $\mathrm{(I_a)},\mathrm{(I_b)},\mathrm{(I_c)},\mathrm{(II)}$ or their inverses. 
\end{remark}

One can prove the following. See \cite{Wada} for a proof.

\begin{proposition}
Any Wirtinger presentation of G(K) can be transformed 
to any other Wirtinger presentation of $G(K)$ 
by an application of a finite sequence of the Tietze transformations $\mathrm{(I_a)},\mathrm{(I_b)},\mathrm{(I_c)},(\mathrm{II})$ and their inverses.
\end{proposition}

By applying the above proposition and the same arguments in the section 3, 
one has the following.

\begin{proposition}
For any $K$,  
$\Delta_{K,\rho}(t)$ defined by a Wirtinger presentation of $G(K)$, 
it is an invariant of $(G(K),\rho)$ up to $t^{2s}\ (s\in\Z)$.
\end{proposition}

\begin{remark}
\noindent
\begin{itemize}
\item
The above holds up to $\pm t^{ls}$ for an $l$-dimensional representation. 
\item
On the other hand, 
by using only the theory of Reidemeister torsion, 
without the arguments in Tietze transformations,  
we can see $\Delta_{K,\rho}(t)$ is well-defined up to $t^{2s}\ (s\in\Z)$.
\end{itemize}
\end{remark}

In general the twisted Alexander polynomial 
$\Delta_{K,\rho}(t)$ depends on a representation $\rho$. 
However the following proposition can be proved easily. 

\begin{definition}
Two representations $\rho,\rho':G(K)\rightarrow \SL(2;\F)$ are called to be conjugate if there exists $P\in \SL(2;\F)$ such that $\rho(x)=P\rho'(x) P^{-1}$ for any $x\in G(K)$. 
\end{definition}

\begin{proposition}
If two representations $\rho$ and $\rho'$ are conjugate, 
then it holds 
$\Delta_{K,\rho}(t)=\Delta_{K,\rho'}(t)$ up to $t^s$. 
\end{proposition}

\begin{example}
If $K$ is the trivial knot, we can take the presentation 
as $
G(K)=\langle x\rangle 
$ 
and the abelianization 
$\alpha:\langle x\rangle\ni x\mapsto t\in\langle t\rangle$. 
In this case, 
any representation $\rho:G(K)\rightarrow \SL(2;\C)$ is given by 
just one matrix $X=\rho(x)\in\SL(2;\C)$.
By definition, one has
\[
\begin{split}
\Delta_{K,\rho}(t)
&=\frac{1}{\det (t\rho(x)-I)}\\
&=\frac{1}{(\lambda_{1}t-1)(\lambda_{2}t-1)}
\end{split}
\]
where $I=\begin{pmatrix} 1 & 0\\ 0& 1\end{pmatrix}$ the identity matrix, 
and $\lambda_{1},\lambda_{2}$ are the eigenvalues of $\rho(x)$.
\end{example}

\begin{example}
Let $\rho={\mathbf 1}:G(K)\ni x\mapsto\begin{pmatrix}1 & 0\\ 0& 1\end{pmatrix}\in \SL(2;\C)$ 
be a 2 dimensional trivial representation. 
Then 
\[
{\mathbf 1}\otimes\alpha=\alpha\oplus\alpha:G(K)\ni x\mapsto 
\begin{pmatrix} 
\alpha(x) & 0\\
0 & \alpha(x)
\end{pmatrix}
\in \GL(2;\C(t)).
\]
Hence it can be seen  
\[
\begin{split}
\Delta_{K,{\mathbf 1}}(t)
&=\frac{\Delta_{K}(t)}{t-1}\cdot \frac{\Delta_{K}(t)}{t-1}\\
&=\left(\frac{\Delta_{K}(t)}{t-1}\right)^2
\end{split}
\]
\end{example}

\begin{example}
Let $\rho_a:G(K)\ni x\mapsto\begin{pmatrix}a & 0\\ 0& a^{-1}
\end{pmatrix}\in \SL(2;\C)(a\in\C\setminus\{0\})$ 
be an abelian representation. 
By direct computation, 
one has   
\[
\begin{split}
\Delta_{K,\rho_a}(t)
&=\frac{\Delta_{K}(at)}{at-1}
\cdot 
\frac{\Delta_{K}(a^{-1}t)}{a^{-1}t-1}\\
&=\left(
\frac{\Delta_{K}(at)}{t-a}\right)\left(\frac{\Delta_{K}(a^{-1}t)}{t-a^{-1}}\right).
\end{split}
\]
Therefore we obtain
\[
\begin{split}
\lim_{a\to 1}\Delta_{K,\rho_a}(t)
&=\Delta_{K,{\mathbf 1}}(t)\\
&=\left(\frac{\Delta_{K}(t)}{t-1}\right)^2.
\end{split}
\]
\end{example}

From these above examples, 
a twisted Alexander polynomial is not a polynomial in general. 
%For any abelian representation $\rho:G(K)\rightarrow \SL(2;\F)$, 
%$\Delta_{K,\rho}(t)$ is not a Laurent polynomial. 

However, under a mild assumption on $\rho$, 
the twisted Alexander polynomial is a Laurent polynomial. 

\begin{proposition}[Kitano-Morifuji\cite{Kitano-Morifuji05}]
If $\rho:G(K)\to \SL(2;\F)$ is not an abelian representation, 
then 
$\Delta_{K,\rho}(t)$ is 
a Laurent polynomial with coefficients in $\F$. 
\end{proposition}

%%%%%%%%%%%%%%%%%%
\subsection{Figure-eight knot}

Let us see the figure-eight knot $4_{1}$ again. 
The knot group $G(4_{1})$ has a presentation as 
\[
G(4_1)=\langle x,y\ |\ wx=yw \rangle
\ (w=x^{-1}yxy^{-1}).
\]

\begin{remark}
Here the generators $x$ and $y$ are conjugate by $w$. 
This is the point to treat $\SL(2;\C)$-representations for a 2-bridge knot. 
\end{remark}

For simplicity, we write $X$ to $\rho(x)$ for $x\in G(K)$. 
The next lemma can be seen by elementary arguments of the linear algebra. 
\begin{lemma}
Let $X,Y\in\SL(2,\C)$. If $X$ and $Y$ are conjugate and $XY\neq YX$, 
then there exists $P\in\SL(2;\C)$ such that 
\[
PXP^{-1}=\begin{pmatrix} s & 1\\0 & 1/s\end{pmatrix},\ 
PYP^{-1}=\begin{pmatrix} s & 0\\u & 1/s\end{pmatrix}.
\]
\end{lemma}

For any irreducible representation $\rho$, 
we may assume that its representative of the conjugacy class 
which contains $\rho$ is given by
\[
\rho_{s,u}:G(4_1)\rightarrow \SL(2;\C)
\]
such that 
\[
\begin{split}
\rho_{s,u}(x)&=X=\begin{pmatrix} s & 1\\0 & 1/s\end{pmatrix},\\
\rho_{s,u}(y)&=Y=\begin{pmatrix} s & 0\\u & 1/s\end{pmatrix}
\end{split}
\]
where 
$s,u\in \C\setminus\{0\}$ . 

\begin{remark}
Because 
\[
\tr (X)=s+\frac{1}{s},\ \tr (X^{-1}Y)=2-u,
\]
then it is seen that the space of the conjugacy classes of the irreducible representations can be parametrized by the traces of $X,X^{-1}Y$.
\end{remark}

We compute the matrix 
\[
R=WX-YW=\rho(w)\rho(x)-\rho(y)\rho(w)
\]
to get the defining equations 
of the space of the conjugacy classes of irreducible representations.

One has each entry of $R=(R_{ij})$:
\begin{itemize}
\item
$R_{11}=R_{22}=0$,
\item
$R_{12}=3-\frac{1}{s^2}-s^2-3 u+\frac{u}{s^2}+s^2 u+u^2$,
\item
$R_{21}=-3 u+\frac{u}{s^2}+s^2 u+3 u^2-\frac{u^2}{s^2}-s^2 u^2-u^3=-uR_{12}$.
\end{itemize}
Hence $R_{12}=0$ is the equation of the space of conjugacy classes of the irreducible representations. 
 
This equation 
\[
3-\frac{1}{s^2}-s^2-3 u+\frac{u}{s^2}+s^2 u+u^2=0
\]
can be solved in $u$:
\[
u=\frac{-1+3 s^2-s^4\pm \sqrt{1-2 s^2-s^4-2 s^6+s^8}}{2 s^2}.
\]

By applying $\displaystyle\frac{\partial}{\partial y}$ to $wx-yw$, 
one has
\[
\begin{split}
\frac{\partial (wx-yw)}{\partial y}
&=\frac{\partial w}{\partial y}-1-y\frac{\partial w}{\partial y}\\
&=(1-y)\frac{\partial w}{\partial y}-1\\
&=(1-y)\frac{\partial}{\partial y}(x^{-1}yxy^{-1})-1\\
&=(1-y)(x^{-1}-wx)-1.
\end{split}
\]
Therefore we obtain 
\[
\begin{split}
A_{\rho,1}
&=\Phi\left(\frac{\partial (wx-yw)}{\partial y}\right)\\
&=(\Phi(1)-\Phi(y))(\Phi(x^{-1})-\Phi(w)\Phi(x))-\Phi(1)\\
&=(I-tY)(t^{-1}X^{-1}-tWX)-I.
\end{split}
\]
Note that $\Phi(w)=W$ because $\alpha(w)=1$. 

Substituting 
\[
u=\frac{-1+3 s^2-s^4\pm \sqrt{1-2 s^2-s^4-2 s^6+s^8}}{2 s^2}
\] 
to each entry and doing direct computations,  
the numerator is given as 
\[
\det A_{\rho,1}=
\frac{1}{t^2}
-\frac{3}{s t}-\frac{3 s}{t}
+6+\frac{2}{s^2}+2 s^2
-\frac{3 t}{s}-3 s t+t^2
\]
Remark that it does not depend on two choices of $u$. 

On the other hand, one has 
\[\det(tX-I)=t^2-\left(s+\frac1s\right)t+1.
\]

Finally we obtain 
\[
\begin{split}
\Delta_{4_1,\rho_{s,u}}(t)
&=\frac{\det A_{\rho,1}}{\det(tX-I)}\\
&=\frac{1}{t^2}-\frac{2 \left(1+s^2\right)}{s t}+1\\
&=\frac{1}{t^2}\left(t^2-2\left(s+\frac{1}{s}\right) t+1\right)\\
&=\frac{1}{t^2}(t^2-2(\tr (X)) t+1).
\end{split}
\]

\begin{remark}
We mention two things. 
The reason for the second one is explained in section 7.
\begin{itemize}
\item
$\Delta_{4_1,\rho_{s,u}}(t)$ is a Laurent polynomial 
because $\rho_{s,u}$ is not abelian. 
\item
$\Delta_{4_1,\rho_{s,u}}(t)$ is monic (explain later) 
because $4_1$ is fibered.
\end{itemize}
\end{remark}

%%%%%%%%%%%%%%%
\subsection{Torus knots}

We can consider that 
$\D_{K,\rho}(t)$ 
is a Laurent polynomial (up to some powers of $t$) valued function 
on the space of conjugacy classes 
of $\SL(2;\C)$-irreducible representations. 
In general a twisted Alexander polynomial is not constant on this space. 
For example, in the case of the figure-eight knot as we discussed above, 
it is depending on the trace of the image of the meridian.

On the other hand, the following holds 
for a $(p,q)$-torus knot $T(p,q)\subset S^3$. 

\begin{theorem}[Kitano-Morifuji\cite{Kitano-Morifuji12}]
For any $(p,q)$-torus knot $T(p,q)$, 
$\Delta_{T(p,q),\rho}(t)$ is a locally constant function 
on each connected component 
of the space of conjugacy classes of $\SL(2;\C)$-irreducible representations. 
\end{theorem}

Let $G(p,q)=\langle x,y\ |\ x^{p}=y^{q}\rangle$ be the knot group of $T(p,q)$. 
Let $m\in G(p,q)$ be the meridian given by $x^{-r}y^{s}$ where $ps-qr=1$ 
and 
$z=x^{p}=y^{q}$ a center element of the infinite order. 
Now let $\rho:G(p,q)\rightarrow \SL(2;\C)$ be an irreducible representation. 

Recall that the center of $\SL(2;\C)$ is $\{\pm I\}$. 
Hence one has $Z=\rho(z)=\pm I$ by the irreducibility of $\rho$.
Then this implies 
\[
X^{p}=\pm I, Y^{q}=\pm I.
\]

Here we may choice 
the eigenvalues of X and Y as 
\begin{itemize}
\item
$\lambda^{\pm 1}=e^{\pm\sqrt{-1}\pi a/p}$ such that $0<a<p$, 
\item
$\mu^{\pm 1}=e^{\pm\sqrt{-1}\pi b/q}$ such that $0<b<q$.
\end{itemize}
Now we get 
\[
\tr (X)=2\cos \frac{\pi a}{p},
\tr (Y)=2\cos\frac{\pi b}{q},
\]
and further 
\[
X^{p}=(-I)^{a},Y^{q}=(-I)^{b}.
\]

\begin{remark}
In any case one has $X^{2p}=Y^{2q}=I$.
\end{remark} 

\begin{proposition}[Johnson\cite{Johnson}]
Any conjugacy class of irreducible representations is uniquely determined for a given triple of traces 
\[
(\tr (X),\tr (Y),\tr (M))
\]
such that 
\begin{itemize}
\item
$\tr (X)=2\cos \frac{\pi a}{p}$,
\item
$\tr (Y)=2\cos\frac{\pi b}{q}$, 
\item
$Z=(-I)^{a}$, 
\item
$\tr (M)\not= 2\cos\pi(\frac{ra}{p}\pm\frac{sb}{q})$, 
\item
$0<a<p$, $0<b<q$, $a\equiv b$ mod 2,
\item
$r,s\in\Z$ such that $pq-rs=1$.
\end{itemize}
\end{proposition}

\begin{corollary}
\noindent
\begin{itemize}
\item
A pair of $(a,b)$ determines a connected component of conjugacy classes. 
\item
Each connected component of the conjugacy classes can be parametrized 
by $\tr (M)\in
\C\setminus
\left\{2\cos\pi\left(\frac{ra}{p}\pm\frac{sb}{q}\right)\right\}$ 
under fixing $(a,b)$. 
\end{itemize}
\end{corollary}

Here we give a proof that twisted Alexander polynomial is constant on each connected component. 
\begin{proof}

We use this parametrization to compute twisted Alexander polynomials. 
By applying Fox's differential to $r=x^{p}y^{-q}$, 
one has
\[
\frac{\partial r}{\partial x}=1+x+\cdots +x^{p-1}.\]

Remark that $\alpha:G(K)\rightarrow \langle t\rangle$ 
is defined by 
$\alpha(x)=t^q,\alpha(y)=t^p$, and $\alpha(m)=t$. 

By the definition, we obtain
\[
\small
\begin{split}
\Delta_{T(p,q),\rho}(t)
=&\frac{\Phi(\frac{\partial r}{\partial x})}{\Phi(y-1)}\\
=&\frac{\det(I+t^q X\cdots +t^{(p-1)q}X^{p-1})}{\det(t^p Y-I)}\\
=&\frac
{(1+\lambda t^{q}+\cdots+\lambda^{p-1}t^{(p-1)q})
(1+\lambda^{-1} t^{q}+\cdots+\lambda^{-(p-1)}t^{-(p-1)q})}
{1-(\mu+\mu^{-1})t^{p}+t^{2p}}.
\end{split}
\]
Hence it can be seen 
$\Delta_{T(p,q),\rho}(t)$ is determined by 
$(p,q)$ and eigenvalues 
$(\lambda,\mu)=(e^{\sqrt{-1}\pi a/p},e^{\sqrt{-1}\pi b/q})$ 
such that $0<a<p, 0<b<q$. 
This means it cannot be varied locally. 
\end{proof}

Now we consider the case of $(2,q)$-torus knot for the simplicity. 
Here the connected components consists of $\frac{q-1}{2}$ components parametrized by odd integer $b$ with $0<b<q$. 
\begin{theorem}[Kitano-Morifuji\cite{Kitano-Morifuji12}]
Twisted Alexander polynomial of $T(2,q)$ is given by 
$$
\Delta_{T(2,q),\rho_b}(t)
=
\left(t^2+1\right)
\prod_{0< k < q,~k:\mathrm{odd},~k\not= b}
\left(
t^2-\xi_k
\right)
\left(
t^2-\bar{\xi}_k
\right),
$$
where 
$\xi_k=\exp\left(\sqrt{-1}\pi k/q\right)$. 
\end{theorem}

\begin{example}
In particular, for the trefoil knot $3_1=T(2,3)$,
there is just one connected component. 
For any irreducible representation $\rho$, we have 
\[
\begin{split}
\Delta_{K,\rho}(t)
&=
\frac{t^6+1}{t^4-t^2+1}\\
&=
t^2+1.
\end{split}
\]
\end{example}

%%%%%%%%%%%%%%
\subsection{Reidemeister torsion, orders, and an obstruction}

Here we mention the relation of the twisted Alexander polynomial with Reidemeister torsion, an order ideal and an obstruction of a representation. 

For simplicity, we treat a representation over $\C$. 
By taking a tensor product of 
\[
\bar{\alpha}:G(K)\cong\pi_{1}(E(K))\ni x\mapsto \alpha(x)^{-1}\in \langle t\rangle\subset\GL(1;\Z[t,t^{-1}])
\]
and 
\[
\rho:G(K)\cong\pi_{1}(E(K))\rightarrow \SL(2;\C),
\]
we have 
\[
\rho\otimes\bar{\alpha}:G(K)\cong\pi_{1}(E(K))\rightarrow \GL(2;\C[t,t^{-1}])\subset\GL(2;\C(t))\]

Further we can define 
a chain complex $C_{\ast}(E(K);\C(t)^{2}_{\rho\otimes\bar\alpha})$ 
by $\rho\otimes\bar{\alpha}$. 
We assume this chain complex is acyclic, namely, 
all homology groups $H_{\ast}(E(K);\C(t)^{2}_{\rho\otimes\bar\alpha})=0$. 
Here we can define Reidemeister torsion 
\[
\tau_{\rho\otimes\bar\alpha}(E(K))\in\C(t).
\]

Under the acyclicity condition, 
we have the following. 

\begin{theorem}[Kitano\cite{Kitano96}]
Up to $t^{2s}(s\in\Z)$, 
it holds 
\[
\Delta_{K,\rho}(t)=\tau_{\rho\otimes\bar\alpha}(E(K)).
\]
\end{theorem}

More generally by considering a twisted homology 
$H_{*}(E(K);\C[t,t^{-1}]^{l}_{\rho\otimes\alpha})$, 
we can consider an order of $H_{*}(E(K);\C[t,t^{-1}]^{l}_{\rho\otimes\alpha})$, 
which is a generalization of the Alexander polynomial 
as a generator of an order ideal. 
This is corresponding to the numerator of $\Delta_{K,\rho}(t)$ for a Wirtinger presentation. 

Here we do not mention the details that the relation between twisted Alexander polynomials and order ideals. 
Please see \cite{Kirk-Livingston}.

In the last part of this section, 
we explain twisted Alexander polynomial is related to an obstruction 
to deform an representation. 

Here assume $G(K)=\langle x_{1},\cdots,x_{n}\ |\ r_{1},\cdots,r_{n-1}\rangle$ 
is a Wirtinger presentation. 
Let $\rho:G(K)\rightarrow \SL(2;\C)$ be a representation 
with 
$X_{i}=\rho(x_{i})$. 
Put another matrix 
$
\tilde{X_{i}}
=\begin{pmatrix}
aX_{i} & \textbf{b}_{i}\\
\textbf{0} & 1
\end{pmatrix}\in\GL(3;\C)$ 
where 
$a\in\C\setminus \{0\}$ and $\textbf{b}_i\in\C^2$.

Now we consider the next problem. 

\begin{problem}
When the map $\tilde{\rho}_{a}:\{x_1,\dots,x_n\}\ni x_{i}\mapsto\tilde{X_{i}}$ gives a representation 
$\tilde{\rho}_{a}:G(K)\rightarrow \GL(3;\C)$ ?
\end{problem}

We can generalize the theorem by de Rham as follows. 
As a generalization of the theorem by de Rham, one has the following. 
 
\begin{theorem}[Wada, unpublished]
Assume $a$ is not an eigenvalue of $X_{1}$. 
Then 
$\tilde{\rho}_{a}:G(K)\rightarrow \GL(3;\C)$ is a representation 
if and only if the numerator of $\Delta_{K,\rho}(a)$ is vanishing. 
\end{theorem}

Hence we can say twisted Alexander polynomial is an obstruction 
to deform a $\GL(2;\C)$-representation 
$G(K)\ni x_i\mapsto aX_i\in \GL(2;\C)$ 
in $\GL(2;\C)\ltimes\C^{2}\subset \GL(3;\C)$.

%%%%%%%%%%%%%%%%%%%%%%%%%%%%
\section{Fibered knot}

A twisted Alexander polynomial is an invariant for $G(K)$ with a representation. 
In general it is not easy to find a linear representation of $G(K)$. 

There are two directions to do it by using a computer. 

\begin{itemize}
\item
a finite quotient (an epimorphism onto a finite group).
\item
a linear representation over a finite field.
\end{itemize}

\subsection{A finite quotient}

If we have a finite quotient, which is an epimorphism onto a finite group $G$:
\[
\gamma:G(K)\rightarrow G.
\]
Here $G$ acts naturally on $G$ and its group rings $\Z G$, $\Q G$.
Then by using $\gamma$, 
$G(K)$ also acts on $G$, $\Z G$ and $\Q G$. 

Note that $\mathrm{dim}_{\Q}(\Q G)=|G|$ 
where $|G|$ is the order of $G$.
Then this gives a $|G|$-dimensional linear representation 
\[
\tilde\gamma:G(K)\rightarrow \GL(|G|;\Q).
\]
Further $\mathrm{Im}\tilde{\gamma}\subset\GL(|G|;\Z)$ 
and  $\mathrm{Im}(\det\circ\tilde{\gamma})=\{\pm 1\}\in\Z$ 
because $G(K)$ acts on $\Z G$. 
Hence the twisted Alexander polynomial 
$\Delta_{K,\tilde\gamma}(t)$ of $K$ is well defined up to $\pm t^s$. 

If $K$ is the trivial knot, 
then a twisted Alexander polynomial has a form of 
\[
\Delta_{K,\rho}=\frac{1}{(\lambda_{1}t-1)\cdots(\lambda_{l}t-1)}
\]
for any $l$-dimensional representation $\rho$. 
Here $\lambda_{1},\cdots,\lambda_{l}$ are eigenvalues of the image of a generator of $E(K)\cong\Z$. 

Now the following holds. 

\begin{theorem}[Silver-Williams\cite{Silver-Wiliams}]
If $K$ is not trivial, then there exists a finite quotient 
$\gamma:G(K)\rightarrow G$ such that 
$\displaystyle
\Delta_{K,\tilde\gamma}(t)
\neq \frac{1}{(\lambda_{1}t-1)\cdots(\lambda_{l}t-1)}$. 
That is, twisted Alexander polynomials distinguish the trivial knot. 
\end{theorem}

%%%%%%%%%
\subsection{Fibered knot}

Recall the definition of a fibered knot. 

\begin{definition}
A knot $K$ is called a fibered knot of genus $g$ 
if $E(K)$ admits a structure of a fiber bundle 
\[
E(K)=S\times [0,1]/(x,1)\sim (\varphi(x),0)
\]
over $S^1$ where $S$ is a compact connected oriented surface $S$ of genus $g$ 
and $\varphi:S\rightarrow S$ is an orientation preserving diffeomorphism.  
\end{definition}

The following classical result is well known. 

\begin{theorem}[Stallings\cite{Stallings}, Neuwirth\cite{Neuwirth}]
A knot $K$ is a fibered knot of genus $g$ 
if and only if 
the commutator subgroup $[G(K),G(K)]$ is a free group of rank $2g$.
\end{theorem}

In general it is not easy to check this condition on $[G(K),G(K)]$. 
The next proposition and its corollary is well known and useful to detect the fiberedness. 
Now we fix a symplectic basis of $H_1(S;\Z)$. 

\begin{proposition}
If $K$ is a fibered knot with a fiber surface $S$ of genus $g$, 
then Alexander polynomial $\Delta_K(t)$ is given by 
\[
\Delta_{K}(t)
=\det (t\varphi_{*}-I:H_{1}(S;\Z)\rightarrow H_{1}(S;\Z))
\]
where $\varphi_{\ast}$ is the induced isomorphism on $H_{1}(S;\Z)$ 
by $\varphi$ and $I$ is the identity matrix of rank $2g$.
\end{proposition}

\begin{corollary}
If $K$ is a fibered knot of genus $g$, then $\Delta_{K}(t)$ is monic and its degree is $2g$. 
\end{corollary}

In general we define the monicness for a Laurent polynomial 
over a commutative ring $R$ 
as follows. 

\begin{definition}
A Laurent polynomial $f(t)$ over $R$ is monic if its coefficient of the highest degree is a unit in $R$. 
\end{definition}

Now we are considering twisted Alexander polynomials of $K$ 
for $\SL(l;\F)$-representations over a field. 
Since any non zero element in a field is always a unit, 
then the above definition of the monicness does not make sense. 
However for any $\SL(n;\F)$-representation, 
twisted Alexander polynomial is well-defined as a rational expression up $\pm t^{s}$. 
Hence we can define the monicness of $\Delta_{K,\rho}(t)$ as follows. 

\begin{definition}
A twisted Alexander polynomial $\Delta_{K,\rho}$ is monic if the highest degree coefficients of the denominator and the numerator are $\pm 1$.
\end{definition}

Generalization to the twisted case is given as follows. 

\begin{theorem}[Cha\cite{Cha}, Goda-Morifuji-Kitano\cite{Goda-Kitano-Morifuji}.]
If $K$ is fibered, then $\Delta_{K,\rho}$ is monic for any $\SL(l,\F)$-representation $\rho$. 
\end{theorem}

If $K$ is fibered, 
then $G(K)$ has the deficiency one presentation 
defined by its fiber bundle structure. 
By using this, it is clear that $\Delta_{K,\rho}(t)$ is monic. 
However it is not clear this presentation can be transformed by strong Tietze transformations. 
In \cite{Goda-Kitano-Morifuji} the above claim was proved for the Reidemeister torsion.

To make refinement of the above results, 
we need the notion of Thurston norm. 
Here the abelianization 
$\alpha:G(K)\rightarrow \Z$ 
can be considered as an integral 1-cocylce on $G(K)$. 
Hence it can be consider as $[\alpha]\in H^1(G(K);\Z)=H^1(E(K);\Z)$. 
Now as one has 
\[
H^1(E(K);\Z)\cong H_2(E(K),\partial E(K);\Z) 
\]
by Poincar\'e duality, 
there exists an properly embedded surface $S=S_1\cup\cdots\cup S_k$ 
whose homology class $[S]$ is dual to $[\alpha]$. 
A surface $S$ may be not connected in general.

Now Thurston norm $||\alpha||_{T}$ is defined by the following. 
\begin{definition}
\[
||\alpha||_T
=\underset{S\subset E(K)}{\mathrm{min}}
\{\ 
\chi_{-}(S)\ 
|\ [S]=\Sigma_i[S_i]\text{ is dual to }
[\alpha]
\}
\]
where 
\[
\begin{split}
\chi_{-}(S)
&=\sum_{i=1}^{k}\mathrm{max}\{-\chi(S_i),0\}\\
&=\sum_{i:\chi(S_{i})<0}-\chi(S_i).
\end{split}
\]
\end{definition}

\begin{example}
If $K$ is a fibered knot of genus $g$, 
then the fiber surface $S$ gives a homology class which is dual to $[\alpha]$. 
Here the euler characteristic $\chi(S)=2-2g-1=1-2g$. 

Hence one has 
\begin{itemize}
\item
$
||\alpha||_T
=2g-1$,
\item
$\mathrm{deg}(\Delta_K(t))=2g$.
\end{itemize}
Therefore we can see 
\[
\begin{split}
||\alpha||_{T}
&=\mathrm{deg}(\Delta_{K}(t))-1\\
&=\mathrm{deg}(\tau_{\alpha}(E(K))
\end{split}
\]
where the degree of $\tau_{\alpha}(E(K))$ is defined 
by $\mathrm{deg}(\Delta_{K}(t))-\mathrm{deg}(t-1)$.
\end{example}

This can be generalized for the twisted Alexander polynomial. 
The next result was turning point to detect the fiberedness of a 3-manifold. 

\begin{theorem}[Friedl-Kim\cite{Friedl-Kim}]
Let $K$ be a fibered knot.  
For any representation $\rho:G(K)\rightarrow\SL(l;\F)$, 
it holds that 
\begin{itemize}
\item
$\Delta_{K,\rho}(t)$ is monic,
\item
$l||\alpha||_T=\mathrm{deg}(\Delta_{K,\rho}(t))$.
\end{itemize}
\end{theorem}

Furthermore the converse is true. 

\begin{theorem}[Friedl-Vidussi\cite{Friedl-Vidussi11-1}]
If the following two conditions hold 
\begin{itemize}
\item
$\Delta_{K,\tilde\gamma}(t)$ is monic, 
\item
$|G|\cdot||\alpha||_T=\mathrm{deg}(\Delta_{K,\tilde\gamma}(t))$, 
\end{itemize}
for any representation $\tilde\gamma:G(K)\rightarrow\GL(|G|;\Q)$ induced by a finite quotient $\gamma:G(K)\rightarrow G$, 
then $K$ is a fibered knot 
and  
the genus of $K$ is given by 
\[
g=\frac{\mathrm{deg}(\Delta_{K,\tilde\gamma}(t))+|G|}{2|G|}.
\]
\end{theorem}

\begin{proof}
Here we explain only outline of the proof of the theorem by Friedl-Vidussi.

Take a Seifert surface 
$S\subset E(K)$  such that $[S]$ is dual to $[\alpha]$ 
and its open neighborhood 
\[
N(S)=S\times (-1,1)
\subset S\times [-1,1]
\subset E(K).\]
Here we consider a submanifold 
\[
M=E(K)\setminus N(K),
\]
which is called a sutured manifold. 

Take a natural inclusion 
\[
\iota:S\rightarrow S\times\{1\}\subset M.
\]
From the condition on twisted Alexander polynomials, 
we can see 
$\iota_\ast:H_{*}(S)\cong H_{*}(M)$
for any twisted coefficient. 

This implies the natural inclusion induces an isomorphism 
\[
\iota_\ast:\pi_{1}S\cong\pi_{1}M.
\]
Therefore we can prove $S\times I\cong M$ 
and $M$ admits a trivial fiber bundle structure over an interval. 
Finally $E(K)$ admits a structure of a fiber bundle over a circle.
\end{proof}

To detect fiberedness, it seems we need to compute Thurston norm $||\alpha||_{T}$. 
In general it is difficult. 
However we do not need to do. 
For a non-fibered knot, we can see the vanishing of a twisted Alexander polynomial. 

\begin{theorem}[Friedl-Vidussi\cite{Friedl-Vidussi}]
If $K$ is not fibered, 
then there exists a representation $\rho$ 
such that $\Delta_{K,\rho}(t)=0$.
\end{theorem}

%%%%%%%%%%%%%%%
\subsection{DFJ-conjecture}

In this subsection we assume that $K$ is a hyperbolic knot. 

\begin{definition}
A knot $K$ is a hyperbolic knot if 
$S^3\setminus K$ admits a complete Riemannian metric of constant sectional curvature -1. 
In other words, $S^3\setminus K$ is the quotient of the three-dimensional hyperbolic space $\mathbb{H}^3$ by a subgroup of hyperbolic isometries 
$\mathit{Isom}_+(\mathbb{H}^3)$ 
acting freely and properly discontinuously. 
\end{definition}

\begin{remark}
It is well known that 
$\mathit{Isom}_+(\mathbb{H}^3)\cong\mathit{PSL}(2;\C)$.
\end{remark}

Let $K$ be a hyperbolic knot. 
Then there exists a holonomy representation 
\[
\bar{\rho}_{0}:G(K)\rightarrow \mathit{PSL}(2;\C)
\]
and a lift  
\[
\rho_0:G(K)\rightarrow 
\SL(2;\C)
\]
with $\tr(\rho_{0}(m))=2$. 
Here $m\in G(K)$ is a meridian.

If $K$ is a fibered knot of genus $g$, then 
twisted Alexander polynomial $\Delta_{K,\rho_{0}}(t)$ is monic polynomial of degree $4g-2$. 

Dunfield, Friedl and Jackson claim it is enough to consider the monicness of $\Delta_{K,\rho_{0}}(t)$ for only $\rho_{0}$ to detect the fiberedness of a hyperbolic knot. 

\begin{conjecture}[Dunfield-Friedl-Jackson\cite{Dunfield-Friedl-Jackson}]
\noindent
\begin{itemize}
\item
$\Delta_{K,\rho_{0}}(t)$ detects Thurston norm of $\alpha$, that is, 
the genus of $K$ can be described by the degree of $\Delta_{K,\rho_0}(t)$. 
\item
A hyperbolic knot $K$ is fibered if and only if $\Delta_{K,\rho_{0}}(t)$ is a monic polynomial. 
\end{itemize}
\end{conjecture}

\begin{theorem}[Dunfield-Friedl-Jackson\cite{Dunfield-Friedl-Jackson}]
DFJ-conjecture is true for all 313,209 hyperbolic knots with at most 15 crossings. 
\end{theorem}

Further it holds for any twist knot. 

\begin{theorem}[Morifuji\cite{Morifuji12}]
DFJ-conjecture is true for any twist knot. 
\end{theorem}

\begin{remark}
\noindent
\begin{itemize}
\item
Morifuji and Tran\cite{Morifuji-Tran} 
treated twisted Alexander polynomials 
of a 2-bridge knot for parabolic representations 
in connection with DFJ-conjecture. 
Here a representation $\rho$ is called a parabolic representation if $\tr(\rho(m))=2$. 
\item
Recently Agol and Dunfield\cite{Agol-Dunfield} showed we can detect the Thurston norm of $K$ by from $\Delta_{K,\rho_{0}}(t)$ in a large class of hyperbolic knots. 
\end{itemize}
\end{remark}

%%%%%%%%%%%%%%
\section{Epimorphism between knot groups}

For the rest of this paper, 
as one application of the twisted Alexander polynomial, 
we treat some topics on epimorphisms between knot groups. 

\begin{definition}
For two knots $K_1, K_2$, 
we write 
$K_1 \geq K_2$ if there exists 
an epimorphism $\varphi :G(K_1)\to G(K_2)$ 
which maps a meridian of $K_1$ to a meridian of $K_2$.
\end{definition}

Let us start from a simple example $8_5\geq 3_1$. 
\begin{example}

%\begin{figure}[htbp]
%\centering
%\includegraphics[width=6cm]{diagram_31.eps}
%{$3_1$}
%\end{figure}

%\begin{figure}[htbp]
%\centering
%\includegraphics[width=6cm]{diagram_85.eps}
%{$8_5$}
%\end{figure}

They have the following presentations:
\begin{align*}
G(8_5)=\langle y_1, y_2, y_3, y_4, y_5, y_6, y_7, y_8 \mid\ &y_7 y_2 y_7^{-1} y_1^{-1},
y_8 y_3 y_8^{-1} y_2^{-1}, y_6 y_4 y_6^{-1} y_3^{-1},\\
&y_1 y_5 y_1^{-1} y_4^{-1},y_3 y_6 y_3^{-1} y_5^{-1}, y_4 y_7 y_4^{-1} y_6^{-1}, \\
&y_2 y_8 y_2^{-1} y_7^{-1}\rangle.
\end{align*}
\begin{align*}
G(3_1)&=\langle x_1 , x_2 , x_3\mid x_3 x_1 x_3^{-1}x_2^{-1}, x_1 x_2 x_1^{-1} x_3^{-1}\rangle.
\end{align*}

If generators of $G(8_{5})$ are mapped to the following generators of $G(3_{1})$ 
as
\[
y_1 \mapsto x_3,\ y_2 \mapsto x_2,\ y_3\mapsto x_1,\ y_4\mapsto x_3,\]
\[
\ y_5\mapsto x_3,\ y_6\mapsto x_2,\ y_7\mapsto x_1,\ y_8\mapsto x_3,\]
any relator in $G(8_5)$ goes to the trivial element in $G(3_1)$. 
For example, it can be seen 
\[
\begin{split}
y_7 y_2 y_7^{-1} y_1^{-1}
&\mapsto x_1 x_2 x_1^{-1} x_3^{-1}=1, \\
y_8 y_3 y_8^{-1} y_2^{-1}
&\mapsto x_3 x_1 x_3^{-1} x_2^{-1}=1.
\end{split}
\]
Hence this gives an epimorphism from $G(8_5)$ onto $G(3_1)$, 
which maps a meridian to a meridian. 
Therefore, we can write $$8_5\geq 3_1.$$
\end{example}

The geometric reason why there exists an epimorphism from $G(8_5)$ to $G(3_1)$ is 
\begin{itemize}
\item
$8_5$ has a period 2, namely, 
it is invariant under some $\pi$-rotation of $S^{3}$, 
\item
$3_1$ is the quotient knot of $8_5$ by this $\pi$-rotation.
\end{itemize}

Here we define a period of a knot as follows. 
\begin{definition}
A knot $K$ in $S^3$ has a period $q>1$ 
if there exists an orientation preserving periodic diffeomorphism  
$f:(S^3,K)\rightarrow (S^3,K)$ of order $q$ 
such that 
the set of fixed points $\textit{Fix}(f)$ is homeomorphic to $S^1$ in $S^3$ 
which is disjoint from $K$. 
 \end{definition}

\begin{remark}
By the positive answer for the Smith conjecture, 
we can see the fixed point set is unknot. 
See \cite{Morgan-Bass} for the Smith conjecture. 
\end{remark}

If $K$ is a periodic knot of order $q$, 
this means there exists an action of $\Z/q\Z$ on $(S^3,K)$. 
Now the quotient space of $S^3$ by this action is topologically $S^3$ 
and 
the image of $K$ by the quotient map is a knot in $S^3$ again. 

The following problem is a fundamental problem.

\begin{problem}
When and how there exists an epimorphism between given knot groups ?
\end{problem}

There are some geometric situations for the existence of a epimorphism as follows. 

\begin{itemize}
\item
To the trivial knot $\bigcirc$ 
from any knot $K$, 
there exists an epimorphism 
\[\alpha:G(K)\rightarrow G(\bigcirc)=\Z.
\]
This is just the {abelianization }
\[G(K)\rightarrow G(K)/[G(K),G(K)]\cong\Z.\]
This map can be always realized a collapse map between knot exteriors with degree one.
\end{itemize}

\begin{itemize}
\item
There exist two epimorphisms 
from any {composite knot} to each of {factor knots}. \newline
\[
G(K_1\sharp K_2)\rightarrow G(K_1), G(K_2).
\]
They are also just induced by {collapse maps} with degree one. 
\item
In general a degree one map between knot exteriors 
induces an epimorphism. 
Explain precisely later.
\item
Let $K$ be a knot with a period $q$.  
Its quotient map $(S^3,K) \rightarrow (S^3,K')=(S^3,K)/_\sim$ induces 
an epimorphism 
\[G(K)\rightarrow G(K').\]
\end{itemize}

\begin{itemize}
\item
For any knot $K$, we take the composite knot $K\sharp \bar{K}$ 
where $\bar{K}$ is the mirror image of $K$.  
The mirror image of $K$ is defined 
as the image of $K$ by a reflection of $K$ along $\R^2$. 
Here we put a knot 
\[
\begin{split}
K\subset &\R^2\times (-\infty,0)\\
\subset &\R^{3}\subset S^{3}=\R^3\cup\{\infty\}.
\end{split}
\] 

This reflection can be naturally extended to $S^3$. 
Then there exist epimorphisms  
\[
G(K\sharp \bar{K})\rightarrow G(K)
\]
between them. 
This epimorphism is induced from a quotient map 
\[(S^3,K \sharp \bar{K})\rightarrow (S^3,K)\] 
of a reflection $(S^3,K\sharp \bar{K})$, whose degree is {zero}. 

\item
There is Ohtsuki-Riley-Sakuma construction for epimorphisms between 2-bridge links. 
Please see \cite{Ohtsuki-Riley-Sakuma} for details. 
\end{itemize}

First we recall the definition of the mapping degree. 

Take any proper map 
\[
\varphi:(E(K_1),\partial E(K_1))\rightarrow (E(K_2),\partial E(K_2))
\]
between two knot exteriors. 
This map $\varphi$ induces a homomorphism 
\[
\varphi_*:H_3(E(K_1),\partial E(K_1);\Z)
\rightarrow 
H_3(E(K_2),\partial E(K_2);\Z).
\]
\begin{definition}
A {degree} of $\varphi$ is defined 
to be the {integer }$d$ 
satisfying  
\[
\varphi_*[E(K_1),\partial E(K_1)]=d[E(K_2),\partial E(K_2)]
\]
where $[E(K_i),\partial E(K_i)]$ is a generator 
of $H_3(E(K_i),\partial E(K_i);\Z)\cong\Z$ 
under the induced orientation from $S^{3}$ for $i=1,2$.
\end{definition}

\begin{proposition}
If $\varphi_*:G(K_1)\rightarrow G(K_2)$ is induced from a degree $d$ map, 
then this degree $d$ can be divisible by the index $n=[G(K_2):\varphi_*(G(K_1))]$. 
Namely {${d}/{n}$ is an integer}. 
\end{proposition}

In particular if $d=1$, then the index $n$ should be 1 
and hence $\varphi_\ast(G(K_1))=G(K_2)$. 
Therefore we obtain the following. 

\begin{corollary}
If there exists a degree one map 
\[
\varphi:(E(K_1),\partial E(K_1))\rightarrow (E(K_2),\partial E(K_2)), \]
then $\varphi$  induces an epimorphism 
\[
\varphi_*:G(K_1)\rightarrow G(K_2).\]
\end{corollary}

\begin{remark}
As explained later, 
there exist epimorphisms induced from 
\begin{itemize}
\item
a non zero degree map, but not degree one map,
\item
a degree zero map
\end{itemize}
\end{remark}

\subsection{Determination on a partial order}

For the set of isomorphism classes of knots, we define a partial order by using epimorphisms. 

\begin{proposition}
The relation $K\geq K'$ gives {a partial order }on the set of the prime knots. 
Namely this relation $\geq$ satisfies the followings;
\begin{enumerate}
\item
$K\geq K$.
\item
$K\geq K',\ K'\geq K\Rightarrow K=K'$.
\item
$K\geq K',\ K'\geq K''\Rightarrow K\geq K''$.
\end{enumerate}
\end{proposition}

\begin{proof}
The only one non trivial claim is the second one,
\[
K\geq K',\ K'\geq K\Rightarrow K=K'.
\]

Here are two facts that we need to prove it. 
\begin{itemize}
\item
Any knot group $G(K)$ is {Hopfian}, 
namely 
any epimorphism $G(K)\rightarrow G(K)$ is an isomorphism. See \cite{Hempel} 
as a reference for example. 
\item
The knot group $G(K)$ determines the knot type for a prime knot $K$ \cite{Gordon-Luecke}.
\end{itemize}

Now we assume $K\geq K',\ K'\geq K$. 
Then there exist two epimorphisms 
$\varphi_{1}:G(K)\rightarrow G(K'),
\ 
\varphi_{2}:G(K')\rightarrow G(K)$.
Here the composition of two epimorphisms $\varphi_{2}\circ\varphi_{1}:G(K)\rightarrow G(K)$ is an isomorphism because $G(K)$ is Hopfian. 
 
Similarly the other $\varphi_{1}\circ\varphi_{2}:G(K')\rightarrow G(K')$ is an isomorphism, too. 
Hence $G(K)$ is isomorphic to $G(K')$. 
Because $K$ and $K'$ are prime knots, then $K=K'$. 
\end{proof}

\begin{remark}
\noindent
\begin{itemize}
\item
To say facts, here we do not use the assumption that an epimorphim preserves a meridian. 
However we need this assumption to determine the partial order. 
\item
Cha and Suzuki\cite{Cha-Suzuki} proved that there exist pairs of knots only with an epimorphism which does not preserve a meridian. 
Namely they admit an epimorphism, but never do an meridian preserving epimorphism. 
\end{itemize}
\end{remark}

To determine partial orders, 
fundamental tools to determine are 
\begin{itemize}
\item
Alexander polynomial,
\item
Twisted Alexander polynomial.
\end{itemize}

The following fact on the Alexander polynomial is well known. As a reference, 
see \cite{Crowell-Fox} for example. 

\begin{proposition}
If $K_1\geq K_2$, then $\Delta _{K_1}(t)$ can be divisible by $\Delta _{K_2}(t)$.
\end{proposition}

This can be generalized to the twisted Alexander polynomial as follows. 

\begin{theorem}[Kitano-Suzuki-Wada\cite{Kitano-Suzuki-Wada}]
If $K_1\geq K_2$ realized by an epimorpshim $\varphi:G(K_{1})\to G(K_{2})$, 
then 
$\Delta_{K_1,\rho _2\circ\varphi}(t)$ can be divisible by $\Delta _{K_2,\rho _2}(t)\ $
for any representation $\rho_2\ :\ G(K_2)\rightarrow \SL(l;\F)$.
\end{theorem}

By using these criterion for $\SL(2;\Z/p\Z)$-representations 
over a finite prime field $\Z/p\Z$, 
we can check the non-existence. 
For the rest, we find epimorphisms between knot groups by using a computer 
and obtain the following list. 

\begin{theorem}[Kitano-Suzuki\cite{Kitano-Suzuki}, Horie-Kitano-Matsumoto-Suzuki\cite{Horie-Kitano-Matsumoto-Suzuki}]
\[
\left.
\begin{array}{l}
8_5,8_{10},8_{15},8_{18},8_{19},8_{20},8_{21},9_1,9_6,9_{16},9_{23},9_{24},9_{28},9_{40},\\
10_5,10_9,10_{32},10_{40},10_{61},10_{62},10_{63},10_{64},10_{65},10_{66},10_{76},10_{77},\\
10_{78},10_{82},10_{84},10_{85},10_{87},10_{98},10_{99},10_{103},10_{106},10_{112},10_{114},\\
10_{139},10_{140},10_{141},10_{142},10_{143},10_{144},10_{159},10_{164}
\end{array}
\right\}\geq 3_1
\]

\[
\left.
\begin{array}{l}
11a_{43}, 11a_{44}, 11a_{46}, 11a_{47}, 11a_{57},11a_{58}, 11a_{71}, 11a_{72}, 11a_{73},\\
11a_{100}, 11a_{106}, 11a_{107},11a_{108}, 11a_{109}, 11a_{117},11a_{134}, 11a_{139}, \\
11a_{157}, 11a_{165},11a_{171}, 11a_{175}, 11a_{176}, 11a_{194}, 11a_{196}, \\
11a_{203}, 11a_{212}, 11a_{216}, 11a_{223}, 11a_{231}, 11a_{232}, 11a_{236}, \\
11a_{244}, 11a_{245}, 11a_{261}, 11a_{263}, 11a_{264}, 11a_{286}, 11a_{305}, 11a_{306}, \\
11a_{318}, 11a_{332}, 11a_{338}, 11a_{340}, 11a_{351}, 11a_{352}, 11a_{355}, \\
11n_{71}, 11n_{72}, 11n_{73}, 11n_{74}, 11n_{75}, 11n_{76}, 11n_{77}, 11n_{78}, 11n_{81}, \\
11n_{85}, 11n_{86}, 11n_{87}, 11n_{94}, 11n_{104}, 11n_{105}, 11n_{106}, 11n_{107}, 11n_{136}, \\
11n_{164}, 11n_{183}, 11n_{184}, 11n_{185}, 
\end{array}
\right\}\geq 3_1
\]

\[
\left.
\begin{array}{l}
9_{18},9_{37},9_{40},9_{58},9_{59},9_{60},10_{122},10_{136},10_{137},10_{138},\\
11a_{5}, 11a_{6}, 11a_{51}, 11a_{132}, 11a_{239}, 11a_{297}, 11a_{348}, 11a_{349}, \\
11n_{100}, 11n_{148}, 11n_{157}, 11n_{165}
\end{array}
\right\}
\geq 4_1
\]

\[
11n_{78}, 11n_{148} \geq 5_1
\]
\[
10_{74},10_{120},10_{122},11n_{71}, 11n_{185}
\geq 5_2
\]
\[
11a_{352} \geq 6_1
\]
\[
11a_{351} \geq 6_2
\]
\[
11a_{47}, 11a_{239} \geq 6_3
\]
\end{theorem}

\subsection{Hasse diagram}

Now let us consider a {Hasse diagram}. 
It is an oriented graph 
for a partial ordering as follows. 

\begin{itemize}
\item
a  {vertex} : a prime knot
\item
an  {oriented ege} : if $K_1\geq K_2$, then we draw it from the vertex of $K_1$ to the one of $K_2$. 
\end{itemize}

Naturally the following problem arises.

\begin{problem}
How can we understand the structure of this Hasse diagram 
of the prime knots under this partial order ?
\end{problem}

By using the Kawauchi's imitation theory\cite{Kawauchi},  
the next theorem can be proved.  
\begin{theorem}[Kawauchi]
For any knot $K$, 
there exists a hyperbolic knot $\tilde{K}$ 
such that 
there exists an epimorphism from $G(\tilde{K})$ onto $G(K)$ 
induced by a {degree one map}. 
\end{theorem}

As a similar application of Kawauchi's theory, 
we can see the following. 

\begin{proposition}
For any knot $K$, there exists a hyperbolic knot $K'$ 
such that 
there exist two epimorphisms from $G(K')$ onto $G(K)$ as follows. 
The one is induced by {degree one map} and another one is induced by {degree zero map}. 
\end{proposition}

From the above proposition, 
there exists an epimorphism from a hyperbolic knot to any knot. 
On the other hand, 
the following fact is known. See \cite{Silver-Whitten06, Kitano-Suzuki08}.
\begin{fact}
For any torus knot $K$, 
if there exists an epimorphism 
$\varphi:G(K)\to G(K')$, 
then $K'$ is a torus  knot, too. 
\end{fact}

Further we can see this Hasse diagram is {not so simple} as follows. 
The following proposition can be also proved by using the Kawauchi's imitation theory. 

\begin{proposition}
For any two prime knots $K_1$ and $K_2$, there exists a prime knot $K$ 
such that $K\geq K_1$ and $K\geq K_2$. 
\end{proposition}

In our list of partial ordering, 
knots 
\[
3_1,4_1,5_1,5_2,6_1,6_2, 6_3\] 
are {minimal} elements 
in the set of prime knots with up to {11-crossings}. 

Here in fact, 
we can prove that they are {minimal} in the set of {all prime knots}. 

\begin{theorem}[Kitano-Suzuki\cite{Kitano-Suzuki14}]
They are minimal elements in the set of all prime knots.
\end{theorem}

%%%%%%%%%%%%%%%%%
By the above results, the following problem appears naturally. 

\begin{problem}
If $K_1\geq K_2$, 
then the {crossing number} of $K_1$ is greater than the one of $K_2$?
\end{problem}

It is clear in the list. 
If it is true in general, 
it gives another proof of the theorem by Agol and Liu. 

\begin{theorem}[Agol-Liu\cite{Agol-Liu}]
Any knot group $G(K)$ {surjects} onto 
only finitely many knot groups.
\end{theorem}

\begin{remark}
This statement was called the Simon's conjecture. See \cite{Kazez}. 
\end{remark}

%%%%%%%%%%%%%%%%%%%%%%%%%%%%%%%%%%%%%%%%%%%%%%%%%%%%%%%%%%%%%%%
\subsection{Epimorphisms induced by degree zero maps}

Boileau, Boyer, Reid and Wang proved the following. 

\begin{proposition}[Boileau-Boyer-Reid-Wang\cite{Boileau-Boyer-Reid-Wang}]
Any {epimorphism} between {2-bridge hyperbolic knots} is always induced from {a non zero degree map}. 
\end{proposition}

On the other hand, there are some interesting examples in our list as follows. 

\begin{example}
Here $10_{59},\ 10_{137}$ are {3-bridge hyperbolic} knots. 
From the list one has 
$10_{59},\ 10_{137}\geq 4_1$, that is, there exist epimorphisms 
\[
G(10_{59}),G(10_{137})\rightarrow G(4_1).
\]
However there is no non-zero degree map between them. 
Namely any epimorphism induced by a proper map between these knot exteriors is induced 
from {a degree zero map}. 
\end{example}
%%%%%%%%%%%%%%%%%%%%%%%%%%

%\begin{figure}[htbp]
%\centering
%\includegraphics[width=6cm]{diagram_1059.eps}
%{$10_{59}$}
%\end{figure}

%%%%%%%%%%%%%%%%%%%%%%%%%%%%%%%

%\begin{figure}[htbp]
%\centering
%\includegraphics[width=6cm]{diagram_10137.eps}
%{$10_{137}$}
%\end{figure}

Here recall the Alexander module of a knot. 
We take a $\Z$-covering 
\[
{E}_\infty(K)\to E(K)
\]
associated to $\alpha:G(K)\rightarrow\Z\cong <t>$. 
Here a group ring $\Z[\Z]\cong\Z[t,t^{-1}]$ acts on $H_1({E}_\infty(K);\Z)$ 
and it gives a structure of a module over a Laurent polynomial ring $\Z[t,t^{-1}]$ 
on $H_1({E}_\infty(K);\Z)$. 
This module is called the Alexander module of $K$ over $\Z$. 

\begin{remark}
If we consider the Alexander module over $\Q$, 
a generator of its order ideal is just Alexander polynomial of $K$. 
\end{remark}

%%%%%%%%%%%%%%%%%%%%%%%%%%%%%%%%%%%%%%%%%%
To see that there are no non-zero degree maps, we have to study the structure of {Alexander modules}.
The following facts are well known 
in the theory of surgeries on compact manifolds. 
For example, see in the book by Wall\cite{Wall}. 

\begin{fact}
If there exits an epimorphism 
\[\varphi_*:G(K)\rightarrow G(K')
\]
induced 
from a non zero degree map\ (resp. a degree one map) 
\[
E(K)\rightarrow E(K'),\] 
then its induced epimorphism 
\[
H_1({E}_\infty(K);\Q)\rightarrow H_1({E}_\infty(K');\Q)
\]
between their Alexander modules over $\Q$
\ (resp. over $\Z$)
is split over $\Q$\ (resp. {$\Z$}). 
\end{fact}

\begin{remark}
The twisted Alexander module version of the above fact 
may be a refinement of the divisibility of twisted Alexander polynomials.
\end{remark}

\begin{example}
By similar observation for Alexander modules, 
we can see the followings.  

\begin{itemize}
\item
$9_{24}\geq 3_1$ and $11a_5\geq 4_1$. 
\item
Any epimorphism induced by a proper map between these knot exteriors is induced 
only from an degree zero map. 
\end{itemize}
\end{example}

\begin{remark}
Here $10_{59},\ 10_{137}, 9_{24}$ are Montesinos knots given as follows.
\begin{itemize}
\item
$10_{59}=M(-1;(5,2),(5,-2),(2,1))$,
\item
$10_{137}=M(0;(5,2),(5,-2),(2,1))$,
\item
$9_{24}=M(-1;(3,1),(3,2),(2,1))$.
\end{itemize}
\end{remark}

%%%%%%%%%%%%%%%%%%%%%%%%%%%%
How there exists an epimorphism between them ?
Recall the geometric observation by {Ohtsuki-Riley-Sakuma} in \cite{Ohtsuki-Riley-Sakuma}. 

Here we assume that 
\[\varphi:G(K)\rightarrow G(K')
\]
is an epimorphism. 

We take a simple closed curve 
${\gamma}\subset S^3\cup K$ 
which belongs to {Ker}$\varphi\subset G(K)$. 
Then if $\gamma$ is an {unknot} 
in $S^3$, by taking the surgery along $\gamma$, 
we get a new knot $\tilde{K}$ in $S^3$ 
such that there exists 
an epimorphism $G(\tilde{K})\rightarrow G(K')$.

%%%%%%%%%%%%%%%%%%%%%%%%%%%%%%%%%%
We can apply this construction to $4_1\sharp \bar{4}_1=4_1\sharp 4_1$. 
First we recall that there exists an epimorphism 
\[
G(4_1\sharp \bar{4_1})\rightarrow G(4_1)
\]
which is a quotient map of a reflection. 
Then it is induced from a degree zero map. 
By surgery along some simple closed curve, 
one has both of 
\[
G(10_{59})\rightarrow G(4_1),\] 
and 
\[
G(10_{137})\rightarrow G(4_1).\] 

More generally we can see the following 
by applying this construction to any 2-bridge knot. 
It was not written explicitly, but essentially in \cite{Ohtsuki-Riley-Sakuma} 
by Ohtsuki, Riley and Sakuma. 

\begin{proposition}
For any 2-bridge knot $K$, 
there exists a Montesinos knot $\tilde{K}$ 
such that 
there exists an epimorphism 
\[G(\tilde{K})\rightarrow G(K)\]
induced from a degree zero map 
$E(\tilde{K})\rightarrow E(K)$. 
\end{proposition}

Return to the list of knots with up to 10-crossings. 
We can find epimorphisms explicitly, but have not found all epimorphisms if there exist. 

For the epimorpshism we could find, 
the following partial order relations can be realized by epimorphisms induced from degree zero maps. \[
\left.
\begin{array}{l}
{8_{10}, \ 8_{20}, \ 9_{24}, \ 10_{62}, \ 10_{65}, \ 10_{77}, }\\
10_{82}, \ 10_{87}, \ 10_{99}, \ {10_{140}, \ 10_{143}}
\end{array}
\right\}
\geq 3_1
\]

\[{10_{59},\ 10_{157}}\geq 4_1
\]

In this list, {Montesinos knots} appear as above.

\begin{remark}
The other knots are given by Conway's notation\cite{Conway} as follows:
\begin{itemize}
\item
$10_{82}=6**4.2,$
\item
$10_{87}=6**22.20,$
\item
$10_{99}=6**2.2.20.20$
\end{itemize}
About the above degree zero maps, 
it might be understood from this classification. 
\end{remark}

%%%%%%%%%%%%
\subsection{Problems}

Finally we put a list of problems. 

\noindent
\begin{itemize}
\item
Characterize a minimal knot in the set of prime knots under the partial order. 
\item
Characterize an epimorphism induced from a degree zero map.
\item
If $K_1,K_2$ are hyperbolic knots and $K_{1}\geq K_{2}$, 
then the hyperbolic volume of $S^3\setminus{K_{1}}$ is greater than or equal to the one of $S^3\setminus{K_{2}}$ ?
\item
How strong is twisted Alexander polynomial for a representation over a finite field ?
\begin{itemize}
\item
To determine the non-existence of an epimorphism.
\item
To detect the fiberedness.
\end{itemize}
For example, is it true that $K$ is fibered 
if any twisted Alexander polynomial is monic for any 2-dimensional unimodular representation 
over a finite prime field ?
\item
By using twisted Alexander module, give a generalization of the method to determine existence of epimorphism by using Alexander module. 
\item
Find skein relation for twisted Alexander polynomial.
\end{itemize}

%%%%%%%%%%%%%%%%%%%%%%%%%%%%

%%%%%%%%%%%%%%%%%%%%

\begin{thebibliography}{999}
\bibliographystyle{amsplain}

\bibitem{Agol-Dunfield}
I. Agol and N. M. Dunfield, 
\textit{Certifying the Thurston norm via $\SL(2, \C)$-twisted homology}, 
arXiv:1501.02136

\bibitem{Agol-Liu}
I. Agol and Y. Liu, 
\textit{Presentation length and Simon's conjecture}, 
J. Amer. Math. Soc. \textbf{25} (2012), no. 1, 151--187. 

\bibitem{Boileau-Boyer}
M. Boileau and S. Boyer, 
\textit{On character varieties, sets of discrete characters and non-zero degree maps}, 
Amer. J. Math. \textbf{134} (2012), no. 2, 285--347.

\bibitem{Boileau-Boyer-Reid-Wang}
M. Boileau, S. Boyer, A. Reid and S. Wang, 
\textit{Simon's conjecture for two-bridge knots}, 
Comm. Anal. Geom. \textbf{18} (2010), 121--143. 

\bibitem{Burde-Zieschang-Heusener}
G. Burde, H. Zieschang and M. Heusener, 
\textit{Knots}, 
Third, fully revised and extended edition. De Gruyter Studies in Mathematics, 
\textbf{5}. De Gruyter, Berlin, 2014. xiv+417 pp. 

\bibitem{Cha}
J. C. Cha, 
\textit{Fibred knots and twisted Alexander invariants}, 
Trans. Amer. Math. Soc. \textbf{355} (2003), no. 10, 4187--4200 

\bibitem{KnotInfo} 
J. C. Cha and C. Livingston, 
KnotInfo: Table of Knot Invariants, http://www.indiana.edu/~knotinfo.%, October 2, 2015.

\bibitem{Cha-Suzuki}
J. C. Cha and M. Suzuki,
\textit{Non-meridional epimorphisms of knot groups}, arXiv:1502.06039

\bibitem{Conway}
J. H. Conway, 
\textit{An enumeration of knots and links, and some of their algebraic properties}, 
in the 1970 Computational Problems in Abstract Algebra (Proc. Conf., Oxford, 1967), 329--358 
Pergamon, Oxford.

\bibitem{Crowell-Fox}
R. Crowell and R. H. Fox, 
\textit{
Introduction to knot theory},  
Reprint of the 1963 original. GTM \textbf{57}. Springer-Verlag, New York-Heidelberg, 1977. 

\bibitem{deRham}
G. de Rham, 
\textit{Introduction aux polyn\^ omes d'uneoud}, 
Enseignement Math. (2) \textbf{13} 1967 187--194 (1968). 

\bibitem{Dunfield-Agol}
N. Dunfield and I. Agol, 
\textit{Certifying the Thurston norm via $\SL(2, \C)$-twisted homology}, 
arXiv:1501.02136

\bibitem{Dunfield-Friedl-Jackson}
N. Dunfield, S. Friedl and N. Jackson, 
\textit{Twisted Alexander polynomials of hyperbolic knots}, 
Exp. Math. \textbf{21} (2012), no. 4, 329--352. 

\bibitem{Fox}
R. H. Fox, 
\textit{Free differential calculus. I. Derivation in the free group ring}, 
Ann. of Math. (2) \textbf{57}, (1953). 547--560. 

\bibitem{Fox-Milnor}
R. H. Fox and J. Milnor, 
\textit{Singularities of 2-spheres in 4-space and cobordism of knots}, 
Osaka J. Math. \textbf{3} (1966), 257--267. 

\bibitem{Friedl-Kim}
S. Friedl and T. Kim, 
\textit{The Thurston norm, fibered manifolds and twisted Alexander polynomials}, 
Topology \textbf{45} (2006), no. 6, 929--953. 

\bibitem{Friedl-Vidussi11-1}
S. Friedl and S. Vidussi, 
\textit{Twisted Alexander polynomials detect fibered 3-manifolds}, 
Ann. of Math. (2) \textbf{173} (2011), no. 3, 1587--1643. 

\bibitem{Friedl-Vidussi11-2}
S. Friedl and S. Vidussi, 
\textit{A survey of twisted Alexander polynomials}, in the book of 
The mathematics of knots, 45--94, 
Contrib. Math. Comput. Sci., 1, Springer, Heidelberg, 2011. 

\bibitem{Friedl-Vidussi}
S. Friedl and S. Vidussi, 
\textit{A vanishing theorem for twisted Alexander polynomials with applications to symplectic 4-manifolds}, 
J. Eur. Math. Soc. \textbf{15} (2013), no. 6, 2027--2041. 

\bibitem{Goda-Kitano-Morifuji}
H. Goda, T. Kitano and T. Morifuji, 
\textit{Reidemeister torsion, twisted Alexander polynomial and fibered knots}, 
Comment. Math. Helv. \textbf{80} (2005), no. 1, 51--61.

\bibitem{Gordon-Luecke}
C. Gordon and J. Lucke, 
\textit{Knots are determined by their complements}, 
J. Amer. Math. Soc. \textbf{2} (1989), no. 2, 371--415. 

\bibitem{Hempel}
J. Hempel, 
\textit{3-manifold}, 
Reprint of the 1976 original. AMS Chelsea Publishing, Providence, RI, 2004.

\bibitem{Horie-Kitano-Matsumoto-Suzuki}
K. Horie, T. Kitano, M. Matsumoto and M. Suzuki, 
\textit{A partial order on the set of prime knots with up to 11 crossings}, 
J. Knot Theory Ramifications, 
\textbf{20}, No. 2 (2011), 275--303. 
Erratum: J. Knot Theory Ramifications, \textbf{21} (2012), no. 4, 1292001, 2 pp.

\bibitem{Jiang-Wang}
B. Jiang and S. Wang, 
\textit{Twisted topological invariants associated with representations},
in Topics in knot theory (Erzurum, 1992), NATO Adv. Sci. Inst. Ser. C Math.
Phys. Sci. 399, Kluwer Acad. Publ., Dordrecht (1993). 

\bibitem{Johnson}
D. Johnson, 
\textit{A geometric form of Casson's invariant and its connection to Reidemeister torsion}, 
unpublished lecture notes. 

\bibitem{Kawauchi}
A. Kawauchi, 
\textit{An imitation theory of manifolds}, 
Osaka J. Math. \textbf{26} (1989), no. 3, 447--464. 

\bibitem{Kazez}
W. H. Kazez (ed.), 
\textit{
Geometric topology},  
Proceedings of the 1993 Georgia International Topology Conference held at the University of Georgia, Athens, GA, August 2-13, 1993. 
AMS/IP Studies in Advanced Mathematics, 2.2. American Mathematical Society, Providence, RI; International Press, Cambridge, MA, 1997. xiv+473 pp. 

\bibitem{Kitano96}
T. Kitano, 
\textit{Twisted Alexander polynomial and Reidemeister torsion}, 
Pacific J. Math. \textbf{174} (1996), no. 2, 431--442. 

\bibitem{Kitano-Morifuji05}
T. Kitano and T. Morifuji, 
\textit{Divisibility of twisted Alexander polynomials and fibered knots}, 
Ann. Sc. Norm. Super. Pisa Cl. Sci. (5) \textbf{4} (2005), no. 1, 179--186.

\bibitem{Kitano-Morifuji12}
T. Kitano and T. Morifuji, 
\textit{Twisted Alexander polynomials for irreducible $\SL(2,\C)$-representations of torus knots}, Ann. Sc. Norm. Super. Pisa Cl. Sci. (5) \textbf{11} (2012), no. 2, 395--406.

\bibitem{Kitano-Suzuki}
T. Kitano and M. Suzuki, 
\textit{A partial order in the knot table}, 
Experiment. Math. \textbf{14} (2005), no. 4, 385--390. 
Erratum: Experiment. Math. 20 (2011), no. 3, 371. 

\bibitem{Kitano-Suzuki08}
T. Kitano and M. Suzuki, 
\textit{
Twisted Alexander polynomials and a partial order on the set of prime knots}, 
in Groups, homotopy and configuration spaces, 307--321, 
Geom. Topol. Monogr., \textbf{13}, Geom. Topol. Publ., Coventry, 2008. 

\bibitem{Kitano-Suzuki14}
T. Kitano and M. Suzuki, 
\textit{Some minimal elements for a partial order of prime knots}, 
arXiv:1412.3168

\bibitem{Kitano-Suzuki-Wada}
T. Kitano, M. Suzuki and M. Wada, 
\textit{Twisted Alexander polynomials and surjectivity of a group homomorphism}, 
Algebr. Geom. Topol. \textbf{5} (2005), 1315--1324. 
Erratum: Algebr. Geom. Topol. \textbf{11} (2011), 2937--2939

\bibitem{Kirk-Livingston}
P. Kirk and C. Livingston, 
\textit{Twisted Alexander invariants, Reidemeister torsion, and Casson-Gordon invariants}, 
Topology \textbf{38}, (1999), no. 3, 635--661.

\bibitem{Lin}
X. S. Lin, 
\textit{Representations of knot groups and twisted Alexander polynomials}, 
Acta Mathematica Sinica, English Series, \textbf{17} (2001), No.3, pp. 361--380

\bibitem{Magnus-Karrass-Solitar}
W. Magnus, A. Karrass and D. Solitar, 
\textit{Combinatorial group theory. 
Presentations of groups in terms of generators and relations}. 
Reprint of the 1976 second edition. Dover Publications, Inc., Mineola, NY, 2004. 

\bibitem{Milnor61}
J. Milnor, 
\textit{Two complexes which are homeomorphic but combinatorially distinct}, 
Ann. of Math. (2) \textbf{74} (1961), 575--590.

\bibitem{Milnor62}
J. Milnor, 
\textit{A duality theorem for Reidemeister torsion}, 
Ann. of Math. (2) \textbf{76} (1962), 137--147. 

\bibitem{Milnor66}
J. Milnor, 
\textit{Whitehead torsion}, Bull. Amer. Math. Soc. \textbf{72} (1966), 358--426.

\bibitem{Milnor67}
J. Milnor, 
\textit{Infinite cyclic coverings}, 
1968 Conference on the Topology of Manifolds (Michigan State Univ., E. Lansing, Mich., 1967), 115--133 Prindle, Weber \& Schmidt, Boston, Mass. 

\bibitem{Morgan-Bass}
J. W. Morgan and H. Bass (ed.),
\textit{The Smith conjecture}, 
Pure and Applied Mathematics, \textbf{112}. Academic Press, 1984. xv+243 pp. 

\bibitem{Morifuji12}
T. Morifuji, 
\textit{On a conjecture of Dunfield, Friedl and Jackson}, 
C. R. Math. Acad. Sci. Paris \textbf{350} (2012), no. 19-20, 921--924. 

\bibitem{Morifuji15}
T. Morifuji, 
\textit{Representation of knot groups into $\SL(2;\C)$ and twisted Alexander polynomials}, 
Handbook of Group Actions (Vol I) (2015), p527--572, Higher Educational Press and International Press, Beijing-Boston.

\bibitem{Morifuji-Tran}
T. Morifuji and A. T. Tran, 
\textit{Twisted Alexander polynomials of 2-bridge knots for parabolic representations}, 
Pacific J. Math. \textbf{269} (2014), no. 2, 433--451. 

\bibitem{Neuwirth}
L. Neuwirth, 
\textit{On Stallings fibrations}, 
Proc. Amer. Math. Soc. \textbf{14} 1963 380--381. 

\bibitem{Ohtsuki-Riley-Sakuma}
T. Ohtsuki, R. Riley and M. Sakuma, 
\textit{Epimorphisms between 2-bridge link groups}, 
in the Zieschang Gedenkschrift, Geom. and Topol. Monogr. \textbf{14} (2008), 417--450.

\bibitem{Rolfsen}
D. Rolfsen, 
\textit{Knots and links}, 
Corrected reprint of the 1976 original. Mathematics Lecture Series, 7. Publish or Perish, Inc., Houston, TX, 1990. xiv+439 pp. 

\bibitem{Riley}
R. Riley, 
\textit{Nonabelian representations of 2-bridge knot groups}, 
Quart. J. Math. Oxford Ser. (2) \textbf{35} (1984), no. 138, 191--208.

\bibitem{Seifert}
H. Seifert, 
\textit{\"Uber das Geschlecht von Knoten}, (German) Math. Ann. \textbf{110} (1935), no. 1, 571--592.

\bibitem{Silver-Whitten06}
D.S. Silver and W. Whitten, 
\textit{Knot group epimorphisms}, 
J. Knot Theory Ramifications \textbf{15} (2006), 153--166.

\bibitem{Silver-Whitten2}
D.S. Silver and W. Whitten, 
\textit{Knot group epimorphisms II}, preprint. 

\bibitem{Silver-Wiliams}
D. Silver and S. G. Wiliams, 
\textit{Twisted Alexander polynomials detect the unknot}, 
Algebr. Geom. Topol. \textbf{6} (2006), 1893--1901. 

\bibitem{Stallings}
J. Stallings, 
\textit{On fibering certain 3-manifolds}, 
Topology of 3-manifolds and related topics (Proc. The Univ. of Georgia Institute, 1961), 
(1962) pp. 95--100 Prentice-Hall, Englewood Cliffs, N.J. 

\bibitem{Wada}
M. Wada, 
\textit{Twisted Alexander polynomial for finitely presentable groups}, 
Topology \textbf{33} (1994), 241--256. 

\bibitem{Wall}
C. T. C. Wall, 
\textit{Surgery on compact manifolds,  Second edition}. 
Edited and with a foreword by A. A. Ranicki. 
Mathematical Surveys and Monographs, 69. AMS, Providence, RI, 1999.
\end{thebibliography}
\end{document}